\documentclass[american]{scrartcl}
\usepackage[T1]{fontenc}
\usepackage[latin9]{inputenc}
\usepackage{babel}
\usepackage{prettyref}
\usepackage{amsthm}
\usepackage{amsmath}
\usepackage{amssymb}
\usepackage{amsfonts}
\usepackage{graphicx}
\usepackage[unicode=true,
 bookmarks=true,bookmarksnumbered=false,bookmarksopen=false,
 breaklinks=false,pdfborder={0 0 1},backref=false,colorlinks=true]{hyperref}
\usepackage{bookmark}

\makeatletter

\theoremstyle{definition}
\newtheorem{defn}{\protect\definitionname}
\theoremstyle{plain}
\newtheorem{thm}{\protect\theoremname}
\theoremstyle{remark}
\newtheorem{rem}{\protect\remarkname}
\theoremstyle{plain}
\newtheorem{prop}{\protect\propositionname}
\theoremstyle{definition}
\newtheorem{example}{\protect\examplename}

\usepackage{dsfont,mathtools}
\usepackage[format=hang]{caption}

\newrefformat{prop}{Proposition \ref{#1}}
\newrefformat{lem}{Lemma \ref{#1}}
\newrefformat{thm}{Theorem \ref{#1}}
\newrefformat{cor}{Corollary \ref{#1}}
\newrefformat{rem}{Remark \ref{#1}}
\newrefformat{example}{Example \ref{#1}}
\newrefformat{def}{Definition \ref{#1}}

\newcommand{\cM}{\mathcal{M}}
\newcommand{\R}{\mathds{R}}

\makeatother

\providecommand{\definitionname}{Definition}
\providecommand{\examplename}{Example}
\providecommand{\propositionname}{Proposition}
\providecommand{\remarkname}{Remark}
\providecommand{\theoremname}{Theorem}

\begin{document}

\title{Do Finite-Size Lyapunov Exponents \\
Detect Coherent Structures?}

\author{Daniel Karrasch\thanks{Email: karrasch@imes.mavt.ethz.ch} \and George Haller\thanks{Corresponding author email: georgehaller@ethz.ch}\\
\and
Institute of Mechanical Systems\\
ETH Zurich, Tannenstrasse 3\\
8092 Zurich, Switzerland}
\maketitle
\begin{abstract}
Ridges of the Finite-Size Lyapunov Exponent (FSLE) field have been
used as indicators of hyperbolic Lagrangian Coherent Structures (LCSs).
A rigorous mathematical link between the FSLE and LCSs, however, has
been missing. Here we prove that an FSLE ridge satisfying certain
conditions does signal a nearby ridge of some Finite-Time Lyapunov
Exponent (FTLE) field, which in turn indicates a hyperbolic LCS under
further conditions. Other FSLE ridges violating our conditions, however, are
seen to be false positives for LCSs. We also find further limitations
of the FSLE in Lagrangian coherence detection, including ill-posedness,
artificial jump-discontinuities, and sensitivity with respect to the
computational time step.
\end{abstract}

\maketitle
\textbf{
Originally developed as a diagnostic for multi-scale mixing,
the Finite-Size Lyapunov Exponent (FSLE) has also been broadly used
to detect coherent structures in dynamical systems. This use of the
FSLE is motivated by a heuristic analogy with the Finite-Time Lyapunov
Exponent (FTLE), a classic measure of particle separation. Here we
derive conditions under which this analogy is mathematically justified.
We also show by examples, however, that the FSLE field has several
shortcomings when applied to coherent structure detection.}

\section{Introduction}

The Finite-Size Lyapunov Exponent (FSLE) is a popular diagnostic of
trajectory separation in dynamical systems. To define this quantity,
one first selects an initial separation $\delta_{0}>0$ and a separation
factor $r>1$ of interest. The separation time $\tau(x_{0};\delta_{0},r)$
is then defined as the minimal time in which the distance between
a trajectory starting from $x_{0}$ and some neighboring trajectory
starting $\delta_{0}$-close to $x_{0}$ first reaches $r\delta_{0}$.
The FSLE $\sigma$ associated with the location $x_{0}$ is then defined
as (cf.\ \cite{Aurell1997,Artale1997,Joseph2002})
\begin{equation}
\sigma\left(x_{0};\delta_{0},r\right)\coloneqq\frac{\log r}{\tau\left(x_{0};\delta_{0},r\right)}.\label{eq:sigmadef-1}
\end{equation}

This quantity infers a local separation exponent for each initial
condition $x_{0}$ over a different time interval of length $\tau(x_{0};\delta_{0},r)$.
The FSLE field is therefore not linked directly to the flow map between
times $t_{0}$ and $t$ for any choice of $t$. In addition, the FSLE
field $\sigma\left(x_{0};\delta_{0},r\right)$ depends on the choice
of the initial separation and the separation factor.

These dependencies are generally viewed as advantages of the FSLE,
enabling the targeted detection of material stretching at different
spatial scales. The spatial average of the FSLE field is particularly
helpful in describing the statistics of trajectory separation under
finite-size perturbations \cite{Cencini2013}.

Beyond Lagrangian statistics, however, the FSLE has also been used
in the detection of specific coherent flow features. In particular,
ridges of the FSLE field have been proposed as indicators of hyperbolic
Lagrangian Coherent Structures (LCSs), which are most repelling or
most attracting material surfaces over a given time interval $[t_{0},t]$ \cite{Joseph2002,dOvidio2004,Bettencourt2013}. This idea is based
on a heuristic analogy with the Finite-Time Lyapunov Exponent (FTLE)
field and an observed visual similarity of the FSLE and FTLE fields, cf.\
\cite[Sec. 10.5.1]{Lai2011} and \cite{Peikert2014}.

Recent results guarantee that certain FTLE ridges do signal nearby
hyperbolic LCSs defined over the \emph{same time interval} $[t_{0},t]$
(cf.\ \cite{Haller2011,Haller2000-1,Farazmand2012,Karrasch2012}).
These results, however, do not extend to ridges of $\sigma\left(x_{0};\delta_{0},r\right)$
in any obvious way, because the latter ridges involve a range of time
scales. Furthermore, assessing rigorously the stability of material
surfaces requires an accurate characterization of the fate of infinitesimally
small perturbations to such surfaces. Using the FSLE in locating LCSs
accurately, therefore, is a diversion from its original mandate, the
description of finite-size perturbations to trajectories.

Here we discuss in detail some marked differences between the FSLE
and FTLE that contradict the broadly presumed equivalence of these
two scalar fields. The differences stem from irregularities of the
FSLE field, which include local ill-posedness, spurious ridges, insensitivity
to changes in the dynamics past the separation time, and intrinsic
jump-discontinuities. Families of such jump-discontinuity surfaces
turn out to be generically present in any nonlinear flow, creating
sensitivity in FSLE computations with respect to the temporal resolution
of the underlying flow data.

We also establish mathematical conditions under which select FSLE
ridges do signal the presence of nearby FTLE ridges, which in turn
mark hyperbolic LCSs under further conditions. The key tool used in
proving this result is a new separation metric, the Infinitesimal-Size
Lyapunov Exponent (ISLE), which we introduce here as the $\delta_{0}\to0$
limit of the FSLE. We also show examples in which FSLE ridges fail
to satisfy our conditions, and indeed do not correspond to nearby
FTLE ridges.

A side-result of our paper is a new ridge definition (cf.\ Definition 2)
that guarantees structural stability for the ridge of a scalar field under
small perturbations. Such a definition has apparently been unavailable in
the literature, and hence should be of independent interest.

\section{Notation and definitions\label{sec:Preliminaries}}

Consider an $n$-dimensional unsteady vector field $v(x,t)$, whose
trajectories are generated by the dynamical system
\begin{align}
\dot{x} & =v(x,t), & x & \in D\subset\mathbb{R}^{n}.\label{eq:velo}
\end{align}
We assume that $v(x,t)$ is of class $C^{3}$ in its arguments. The
trajectory of \eqref{eq:velo} starting from the point $x=x_{0}$
at time $t=t_{0}$ is denoted $x\left(t;t_{0},x_{0}\right)$, which
allows us to define the flow map as
\[
F_{t_{0}}^{t}(x_{0})\coloneqq x\left(t;t_{0},x_{0}\right).
\]
We will use the Cauchy--Green strain tensor $C{}_{t_{0}}^{t}(x_{0}):=\left[DF_{t_{0}}^{t}(x_{0})\right]^{T}DF_{t_{0}}^{t}(x_{0})$,
a symmetric positive definite tensor field associated with the flow
map. The maximal and minimal strain eigenvalues and the corresponding
strongest and weakest unit strain eigenvectors of $C{}_{t_{0}}^{t}(x_{0})$
satisfy the relations
\begin{align*}
C{}_{t_{0}}^{t}(x_{0})e_{\mathrm{min}}\left(C{}_{t_{0}}^{t}(x_{0})\right) & =\lambda_{\mathrm{min}}\left(C{}_{t_{0}}^{t}(x_{0})\right)e_{\mathrm{min}}\left(C{}_{t_{0}}^{t}(x_{0})\right),\\
C{}_{t_{0}}^{t}(x_{0})e_{\mathrm{max}}\left(C{}_{t_{0}}^{t}(x_{0})\right) & =\lambda_{\mathrm{max}}\left(C{}_{t_{0}}^{t}(x_{0})\right)e_{\mathrm{max}}\left(C{}_{t_{0}}^{t}(x_{0})\right).
\end{align*}
The Finite-Time Lyapunov Exponent (FTLE) over the time interval $[t_{0},t]$
is then defined as
\begin{equation}
\Lambda_{t_{0}}^{t}(x_{0})\coloneqq\frac{1}{2\left(t-t_{0}\right)}\log\lambda_{\max}\left(C{}_{t_{0}}^{t}(x_{0})\right).\label{eq:FTLE}
\end{equation}
The FTLE measures the largest average exponential separation rate
between the trajectory starting at $x_{0}$ and trajectories starting
infinitesimally close to $x_{0}$. The separation rates in this maximization
are compared over a common length of time $\left(t-t_{0}\right)$,
and hence describe the stretching properties of the flow map $F_{t_{0}}^{t}$
near $x_{0}$.

An alternative assessment of separation in the flow is provided by
the Finite-Size Lyapunov exponent (FSLE). As already noted in the
Introduction, for a fixed separation factor $r>1$, and initial
separation $\delta_{0}>0$, the separation time $\tau(x_{0};\delta_{0},r)$
is defined as
\begin{equation}
\tau(x_{0};\delta_{0},r)=\min_{\left\vert x_{1}-x_{0}\right\vert =\delta_{0}}\left\{ \left|t-t_{0}\right|:\,\, t>t_{0},\,\,\left\vert F_{t_{0}}^{t}(x_{1})-F_{t_{0}}^{t}(x_{0})\right\vert =r\delta_{0}\right\} ,\label{eq:taudef}
\end{equation}
from which the FSLE field $\sigma(x_{0};\delta_{0},r)$ is computed
as in \eqref{eq:sigmadef-1}.

The only setting in which FSLE and FTLE are directly related by a formula is
that of linear dynamical systems. Such systems exhibit spatially
homogeneous separation properties at all scales, and hence the quantities $\sigma$,
$\tau$ and $\Lambda_{t_{0}}^{t}$ are all independent of the initial
condition $x_{0}$ and the initial separation $\delta_{0}$. Therefore,
for linear systems we obtain
\[
\sigma(r)=\Lambda_{t_{0}}^{t_{0}+\tau(r)}
\]
for the common FSLE and FTLE values that all trajectories share.

For nonlinear systems, however, the separation time $\tau$ will also
depend on the initial condition $x_{0}$ and the initial separation
$\delta_{0}$. As a consequence, the FSLE and FTLE fields will no longer be computable from each other, no matter what integration time is used in the FTLE.

More generally, the philosophy behind computing FSLE is not in line
with classical, observation-driven assessments of flow properties
between fixed initial and final times. All basic concepts in dynamical
systems and Lagrangian continuum mechanics build on properties of
the flow map, and hence fall in the latter observational category.
Specifically, material surfaces that show locally extreme repulsion
or attraction over a fixed observational period (i.e., hyperbolic
LCSs) have no immediate connection with the FSLE field.

\section{Non-equivalence of the FSLE and FTLE fields \label{sec:FSLE-FTLE}}

Despite the above conceptual differences between the FTLE and FSLE,
they are often assumed to be operationally equivalent. Below we give
several reasons why such an equivalence cannot hold in general.

\subsection{Ill-posedness of FSLE}

While the FTLE is well-defined for any choice of its arguments (any
initial condition and integration time), the FSLE is not defined at
$x_{0}$ if the local separation around $x_{0}$ never reaches $r$-times
the initial separation over the duration of the available velocity
data. This affects more and more initial conditions as $r$ is increased.

Consider, for instance, the transient saddle flow
\begin{equation}
\begin{split}\dot{x}_{1} & =-x_{1}+s(t)(1+x_{1}),\\
\dot{x}_{2} & =x_{2}-s(t)x_{2},
\end{split}
\label{eq:ex1}
\end{equation}
with a smooth function $s(t)$ that is strictly monotone increasing
over the time interval $[a,b]$, with $s(t)=0$ for $t\leq a$ and
$s(t)=1$ for $t\geq b>a>0$. The flow \eqref{eq:ex1} represents
a smooth transition from a stagnation point flow to a parallel shear
flow over the time interval $[a,b]$.

Two trajectories of \eqref{eq:ex1} starting from $y_{0}=\left( y_{0}^1,y_{0}^2\right)$
and $x_{0}=\left(x_{0}^1,x_{0}^2\right)$ at time $t_{0}=0$ satisfy the estimate
\begin{align*}
\left|F_{0}^{t}(y_{0})-F_{0}^{t}(x_{0})\right| & \leq\left|y_{0}^1-x_{0}^1\right|+\left|y_{0}^2-x_{0}^2\right|\mathrm{e}^{\int_{0}^{t}(1-s(\tau))\,\mathrm{d}\tau}\\
 & \leq\left|y_{0}^1-x_{0}^1\right|+\left|y_{0}^2-x_{0}^2\right|\mathrm{e}^{b}\leq\left|y_{0}-x_{0}\right|\mathrm{e}^{b}.
\end{align*}
Therefore, for any choice of the initial separation $\delta_{0}=\left|y_{0}-x_{0}\right|$,
the final separation of the trajectories will never be larger than
$\delta_{0}\mathrm{e}^{b}$, and hence the FSLE field $\sigma\left(x_{0};\delta_{0},r\right)$
is undefined for any $r>\mathrm{e}^{b}$. At the same time, the FTLE
field $\Lambda_{0}^{t}(x_{0})$ is well-defined for any choice of
$t$ and $x_{0}$.

In exploring an \emph{a priori} unknown flow, the identification of
the maximal meaningful $r$ value for which the FSLE is well-defined
can be a costly numerical process.

\subsection{Insensitivity of FSLE to later changes in the flow}

Past the separation time $\tau(x_{0};\delta_{0},r)$, the FSLE will
become insensitive to any further changes in stretching rates along
a trajectory $F_{t_{0}}^{t}\left(x_{0}\right)$. By contrast, the
FTLE, when computed over increasing integration times,
keeps monitoring the same trajectory beyond the time $\tau(x_{0};\delta_{0},r)$,
continually revising the averaged largest exponential stretching rate
along $F_{t_{0}}^{t}\left(x_{0}\right)$.

For instance, consider the incompressible flow
\begin{equation}
\begin{split}\dot{x}_{1} & =-x_{1}-s(t)\frac{x_{1}}{\cosh(x_{2})^{2}},\\
\dot{x}_{2} & =x_{2}+s(t)\tanh x_{2},
\end{split}
\label{eq:nonlinear saddle}
\end{equation}
with the smooth function $s(t)$ again defined as in \eqref{eq:ex1}.
This flow turns from a linear saddle into a nonlinear saddle gradually
over the $[a,b]$ time interval. Therefore, computing the FSLE field
$\sigma\left(x_{0};\delta_{0},r\right)$ from $t_{0}=0$ with any
$r\leq \mathrm{e}^{a}$ gives
\[
\sigma\left(x_{0};\delta_{0},r\right)=1,
\]
for any choice of $x_{0}$ and $\delta_{0}$. Therefore, irrespective
of the later transition of the flow from linear to nonlinear, a computation
of the FSLE field will return $\sigma\equiv1$ for a range of $r$
values. This range grows exponentially with the magnitude of $a$.
Again, in case of an\emph{ a priori} unknown flow, one is unaware
of flow structures and their temporal changes, and would precisely
like to use the FSLE to obtain information about these unknown factors.
Exploring all possible choices of the separation factor $r$ to obtain
this information is clearly a tedious procedure.

By contrast, over a fixed time interval of observation $[0,t]$, the
FTLE field $\Lambda_{t_{0}}^{t}(x_{0})$ will correctly assess the
uniform linear separation rate for times up to $t=a$, then starts
reflecting the developing inhomogeneity in separation by highlighting
the stable manifold of the nonlinear saddle as a ridge for times $t>a$
(see Fig.\ \ref{fig:simple}).

\begin{figure}
\centering
\includegraphics[width=0.3\textwidth]{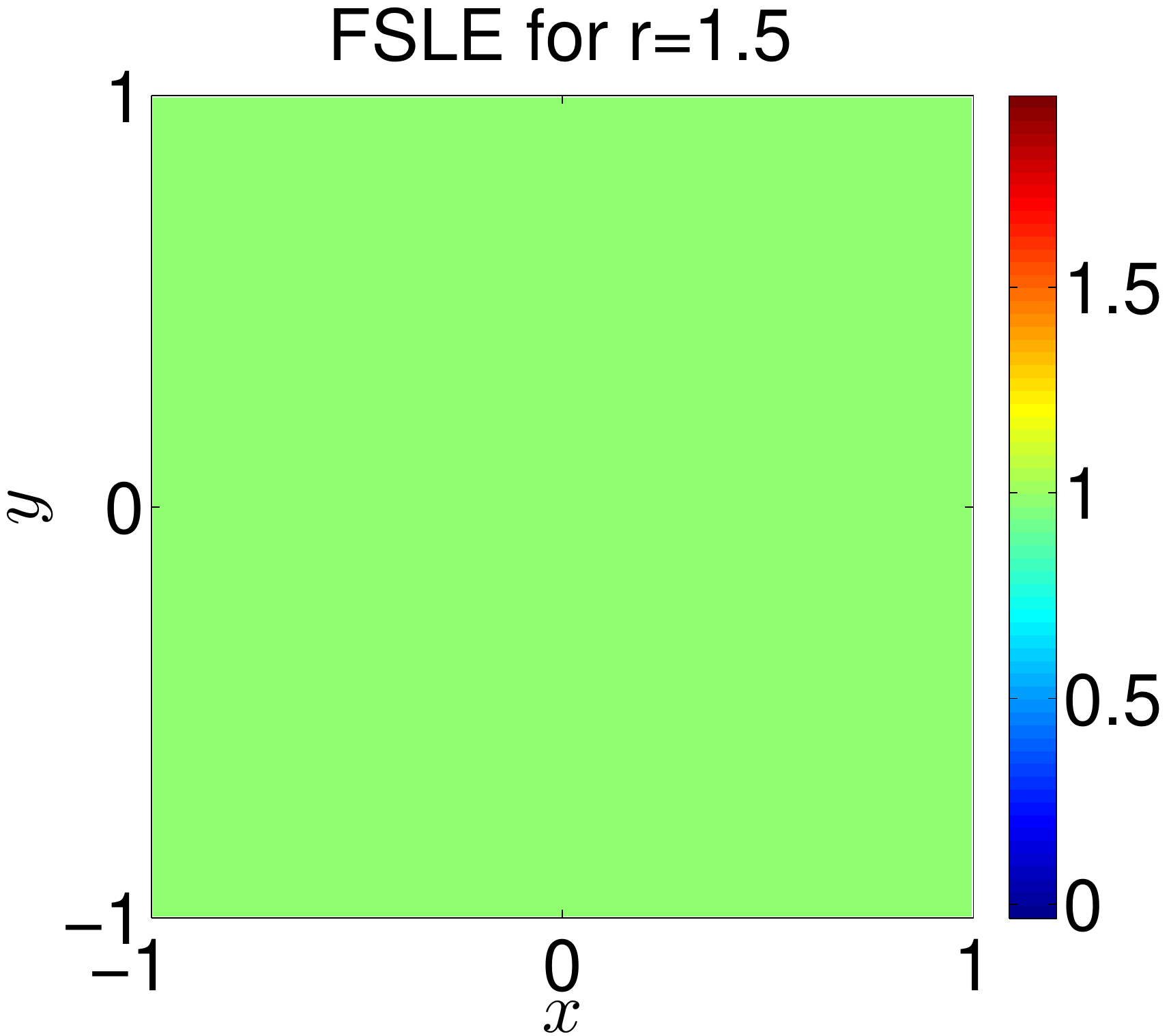}\qquad\includegraphics[width=0.3\textwidth]{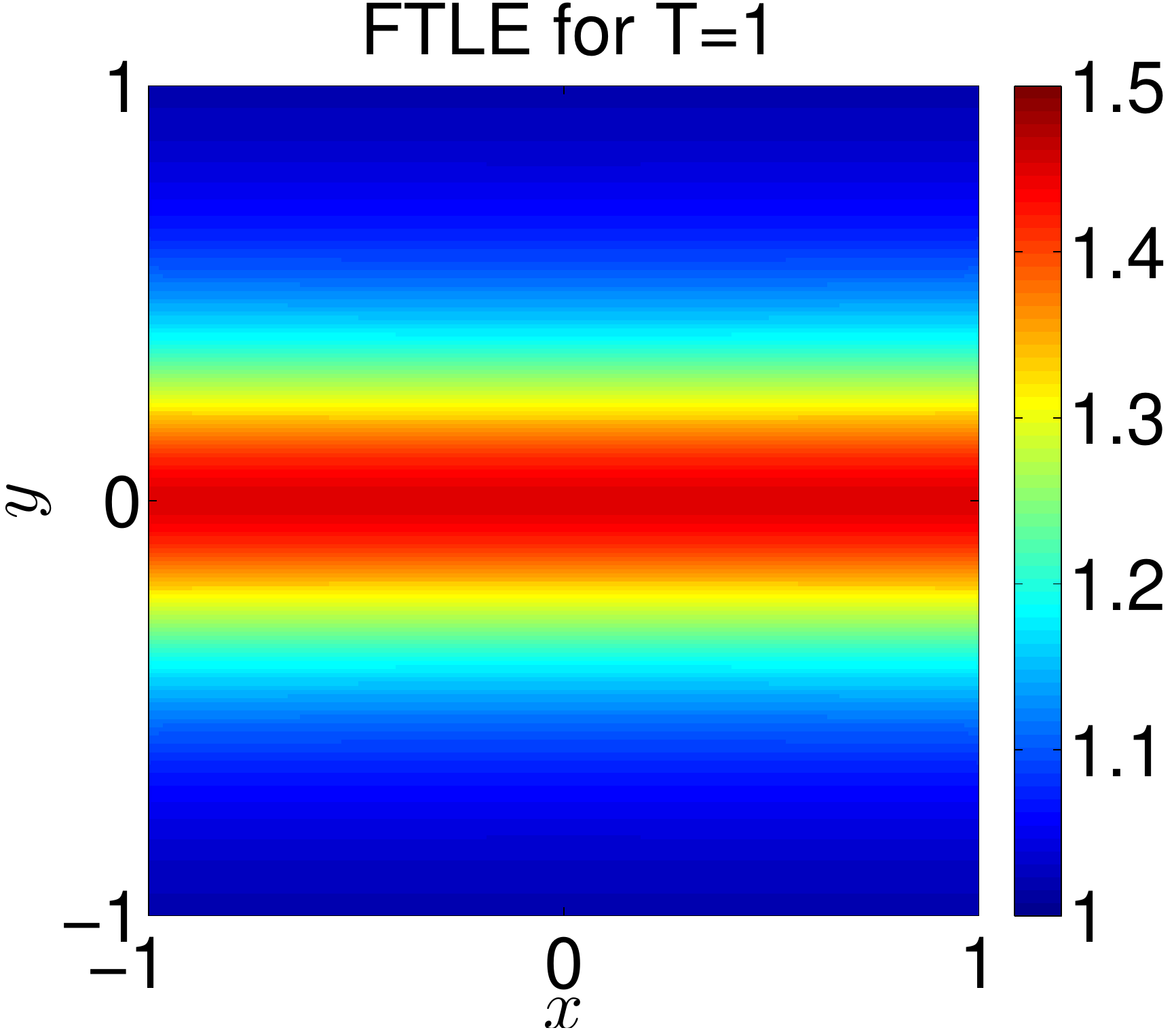}
\caption{FSLE and FTLE fields for the flow \eqref{eq:nonlinear saddle}, with
initial time $t_{0}=0$ and transition between $a=0.5$ and $b=0.6$.
Left: FSLE field for $r=1.5$. Right: FTLE field for $T=1$.}
\label{fig:simple}
\end{figure}

\subsection{Inherent jump-discontinuities of FSLE \label{sec:non-smoothness}}

While the FTLE is defined through the explicit formula \eqref{eq:FTLE},
the FSLE field \eqref{eq:sigmadef-1} relies on the separation time
$\tau$ defined implicitly by Eq.\ \eqref{eq:taudef}. Solutions
of such implicit equations generally admit discontinuities, and the
FSLE field is no exception to this rule.

To illustrate this, we consider the system
\begin{equation}
\begin{split}\dot{x}_{1} & =-0.1\pi\sin\pi x_{1}\cos\pi x_{2},\\
\dot{x}_{2} & =0.1\pi\cos\pi x_{1}\sin\pi x_{2},
\end{split}
\label{eq:gyre-1}
\end{equation}
a specific steady version of the double-gyre flow introduced by \cite{Shadden2005}.

The left plot in Fig.\ \ref{fig:single_gyre} shows the FSLE field
computed over one of the two gyres in the flow with separation factor
$r=6$. Jump-discontinuities along a large family of curves are readily
observed. These discontinuities become even more apparent in the middle
plot, where we graph FSLE values along the line $x_{\textnormal{2,spec}}=0.48$
and $x_{1}\in\left[0,0.5\right]$.

\begin{figure}
\centering
\includegraphics[width=0.35\textwidth]{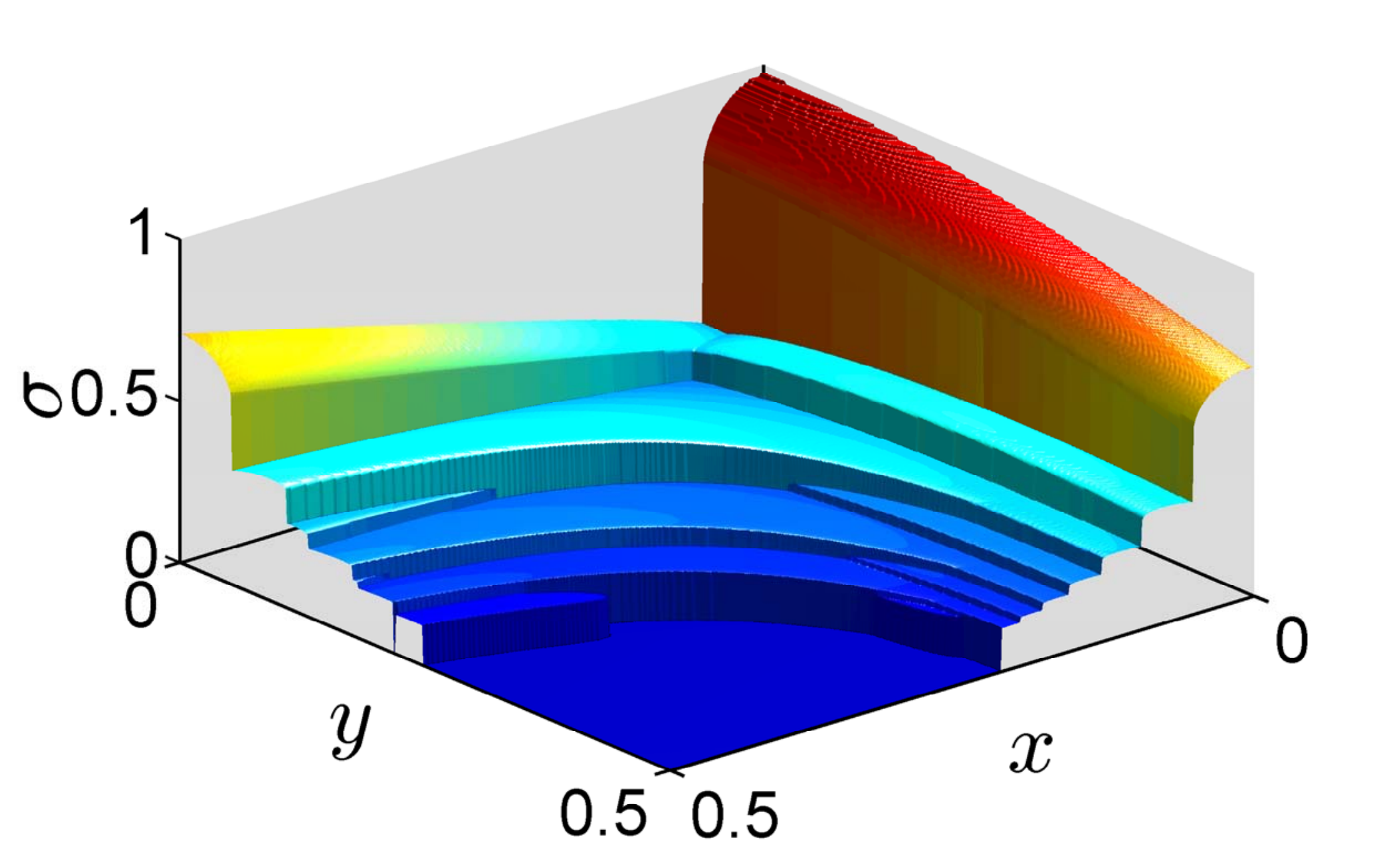}\includegraphics[width=0.3\textwidth]{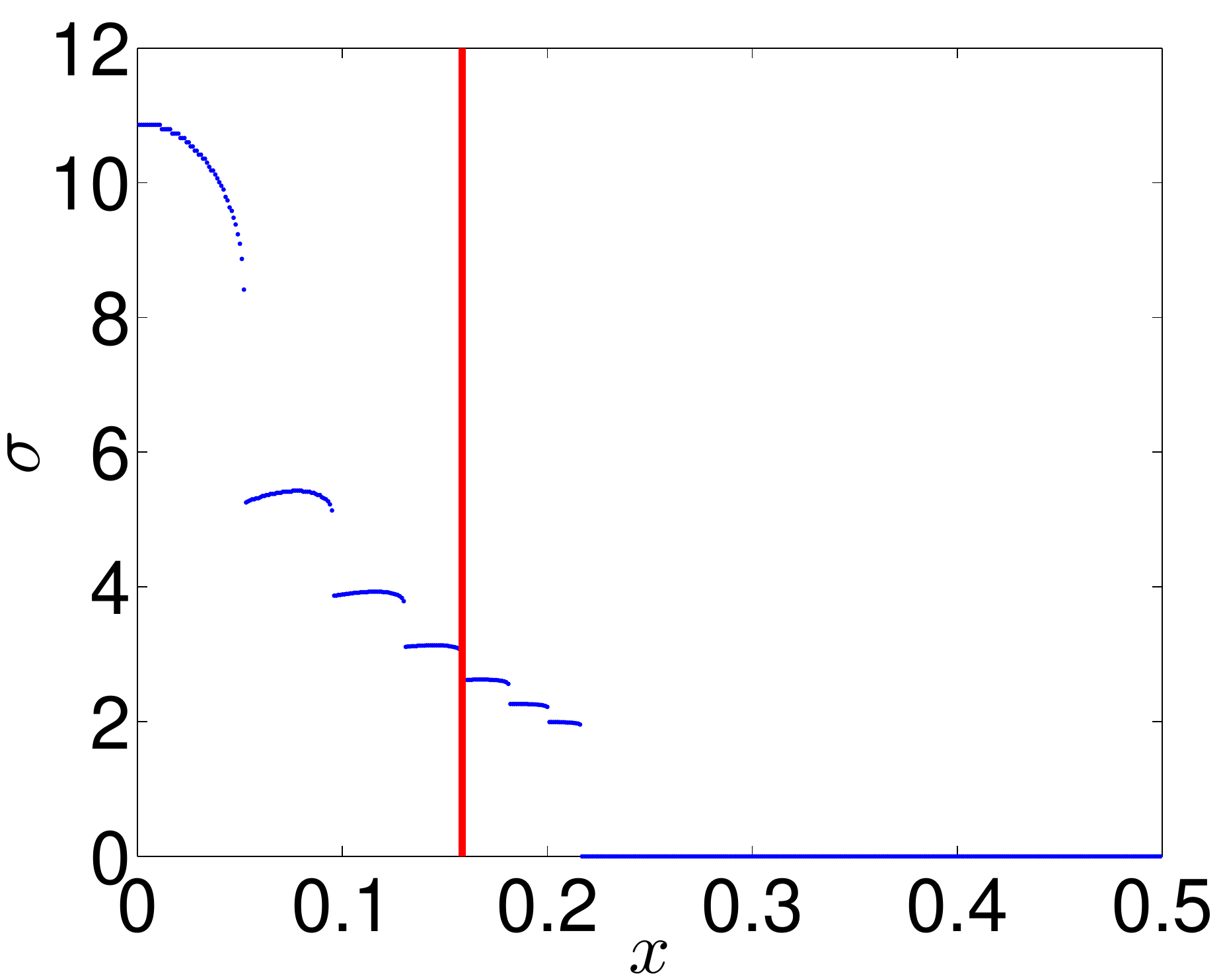}
\includegraphics[width=0.3\textwidth]{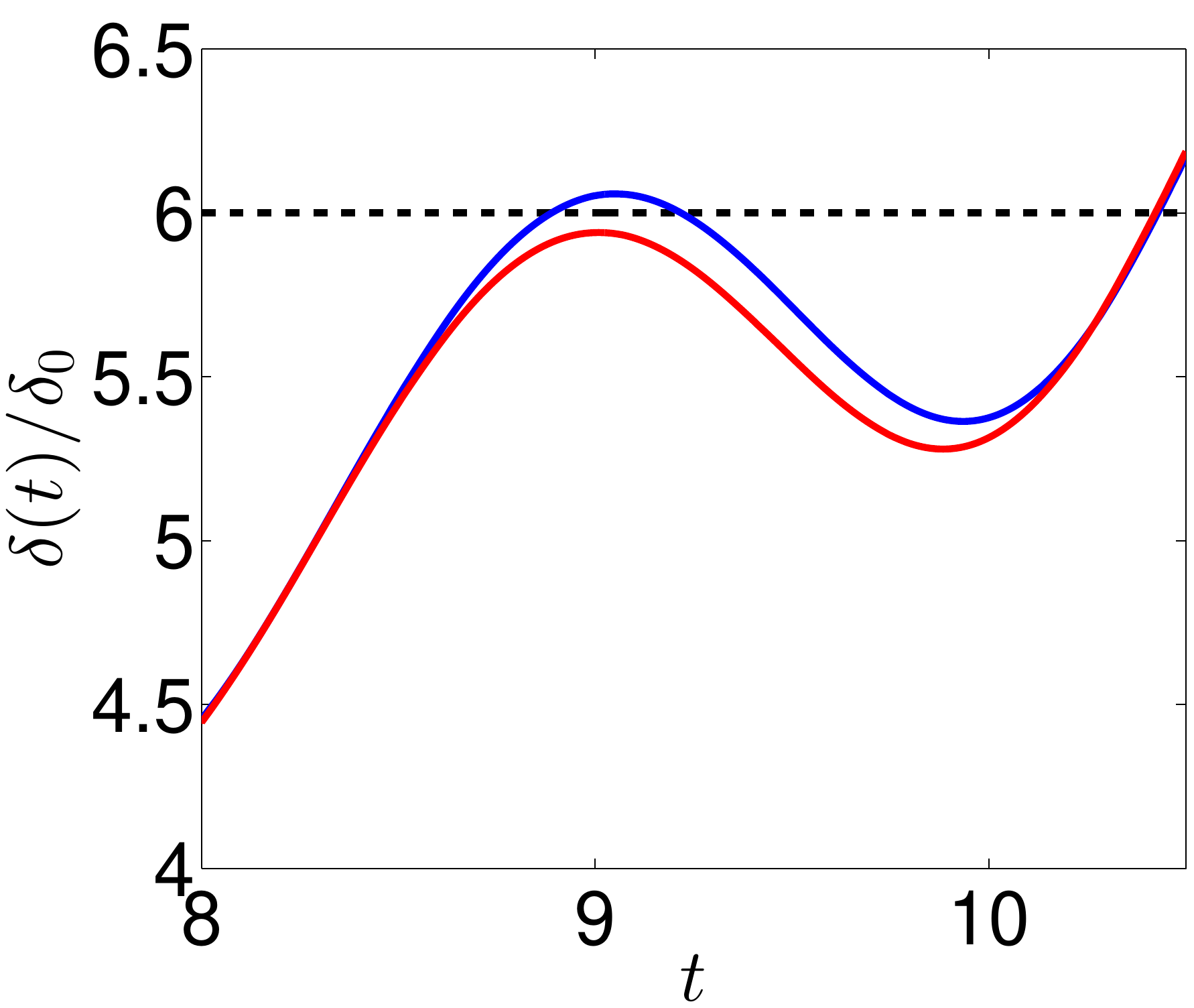}
\caption{Left: Lines of discontinuities in the FSLE field of the autonomous
double gyre flow \eqref{eq:gyre-1}. Middle: A cross section of the
discontinuities along the $x_{2,\textnormal{spec}}=0.48$ line. Right:
The root cause of the jump-discontinuity at $x_{\textnormal{1,dis}}\approx0.1583$:
a tangency of the particle-separation history curve with the $r=6$
horizontal line.}
\label{fig:single_gyre}
\end{figure}

To identify the root cause behind such jump-discontinuities, we focus
on one of the jump locations at $x_{\textnormal{1,dis}}\approx0.1583$,
indicated by the vertical line in the middle plot of Fig.\ \ref{fig:single_gyre}
around which a discontinuity is observed. In the right plot of Fig.
\ref{fig:single_gyre}, we graph the time evolution of the largest
relative particle separation $\delta(t,x_{0})/\delta_{0}$ for two
nearby initial conditions, one on the left and one on the right
of the $x_{\textnormal{1,dis}}\approx0.1583$ line. This plot reveals
a tangency between the $\delta(t;x_{\textnormal{1,dis}},x_{\textnormal{2,spec}})/\delta_{0}$
curve and the $r=6$ horizontal line. A consequence of this tangency
is a sizable jump in the separation time (i.e., the smallest solution
of the equation $\delta(\tau;x_{\textnormal{0}})/\delta_{0}=6$) as
initial conditions are varied across $x_{\textnormal{1,dis}}$.

As we establish later in this paper, the above jump-discontinuities
of the FSLE field are typical. They generically occur along families
of codimension-one surfaces in the phase space of a nonlinear dynamical
system (cf.\ Proposition \ref{prop:degeneracy}).

\subsection{Sensitivity of FSLE field with respect to temporal resolution \label{sec:sensitivity}}

The presence of jump-discontinuities in the FSLE field may appear
to be a strictly cosmetic issue. However, jumps in $\tau(x_{0};\delta_{0},r)$
result in a sensitivity of FSLE calculations with respect to the temporal
resolution of the available flow data.

Indeed, the right subplot of Fig. \ref{fig:single_gyre} shows that
under a course step size in $t$, the first crossing of the $\delta(t)/\delta_{0}=6$
line by the particle-separation history curve will be missed altogether
for an open set of initial conditions. Instead of the correct separation
time, a larger separation time will be recorded. The resulting error
will typically be substantially larger than the time-step used in
the computation.

LaCasce \cite{LaCasce2008} has already observed that FSLE statistics
show sensitivity with respect to the temporal resolution in the range
of smaller $\delta_{0}$ scales. Such sensitivity can be gradually
reduced for analytic and numerical model flows by selecting smaller
and smaller time steps. Step-size reduction, however, is not an option
for \emph{in situ} observational flow data, which comes with a fixed
(and typically course) temporal resolution \cite{LaCasce2008}.

By contrast, the FTLE field is everywhere continuous in the time parameter
$t$ and the initial condition $x_{0}$. Furthermore, errors in an FTLE computation are typically of higher order with respect to the computational time step used in integrating the velocity field.

\subsection{Spurious ridges of FSLE}\label{sec:movesep}

Even in regions where the FSLE field is well-defined and continuous,
FSLE ridges may signal false positives for repelling LCSs. Because
of significant changes in the flow after the separation time is reached
by key trajectories, the FSLE field may even produce such spurious
ridges along \emph{trenches} of the FTLE field.

As an example, consider a two-dimensional model for moving unsteady
separation along a horizontal free-slip wall. The velocity field derives
from the stream-function Hamiltonian
\begin{equation}
H(t,x)=-L\tanh(q_{2}x_{2})\tanh(q_{1}(x_{1}-at)),\label{eq:separationexample}
\end{equation}
where $L$ characterizes the strength of the separation; $q_{1}$
and $q_{2}$ control how localized the impact of separation is on
the flow; and $a$ defines the horizontal speed at which the separation
moves. The flow, therefore, becomes steady in a frame that moves horizontally
with speed $a$.

We fix the parameters $L=4$, $q_{1}=5$, $q_{2}=1$, and $a=10$,
and choose the maximal length of observation time as $T=t-t_{0}=2$.
For this choice of parameters, Fig.\ \ref{fig:failure_vf} shows the
instantaneous velocity field for the model \eqref{eq:separationexample}
at $t=0$.

\begin{figure}
\centering
\includegraphics[width=0.5\textwidth]{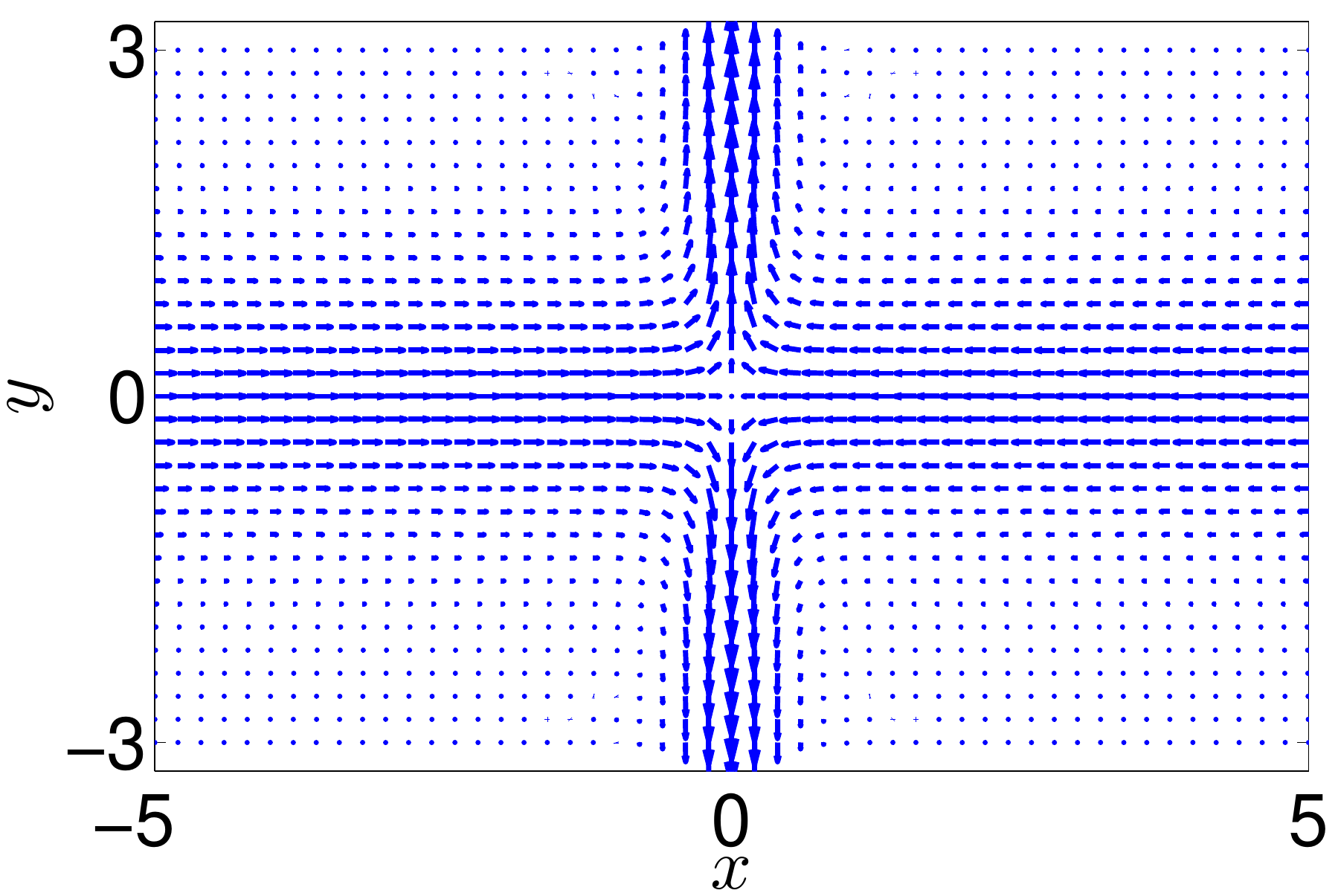}
\caption{Instantaneous velocity field for the moving separation flow \eqref{eq:separationexample}.}
\label{fig:failure_vf}
\end{figure}

Selecting the separation factor as $r=2.3$, we observe in Fig.\ \ref{fig:snow}
(top left) an FSLE ridge along the $x_{1}$-axis, starting at about $0.3$.
Since the separation point moves to the right, initial positions with
larger and larger $x_{1}$-values experience separation at later and
later times. As a result, the height of the FSLE ridge along the $x_{1}$
axis shows large variation. The separation-time plot in Fig.\ \ref{fig:snow}
(top right) highlights this further, indicating a separation-time valley
with increasing bottom-height.

At the same time, by the localized nature of the separation, only
a short segment of the $x_{1}$ axis will generate a ridge for the
FTLE field for any choice of the integration time. This axis segment
is the subset of initial conditions showing the most net separation
over the time interval $T$. The rest of the $x_{1}$-axis is in fact
an FTLE trench, as seen in Fig.\ \ref{fig:snow} (bottom left). For longer
integration times, an increasingly long subset of the $x_{1}$-axis
becomes an FTLE valley, while two nearby FTLE ridges parallel approach
it (Fig.\ \ref{fig:snow}, bottom right). Thus, one cannot even argue
that the FSLE ridge can at least be continued into a nearby, unique
FTLE ridge.

\begin{figure}
\centering
\includegraphics[width=0.49\textwidth]{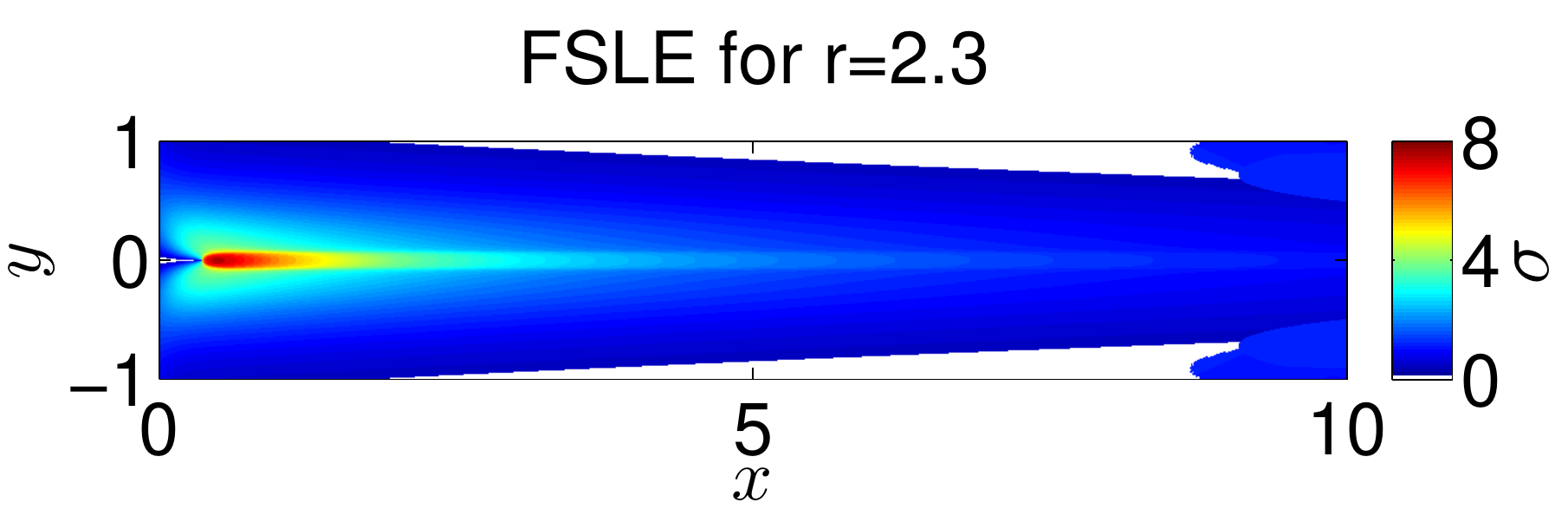}
\includegraphics[width=0.49\textwidth]{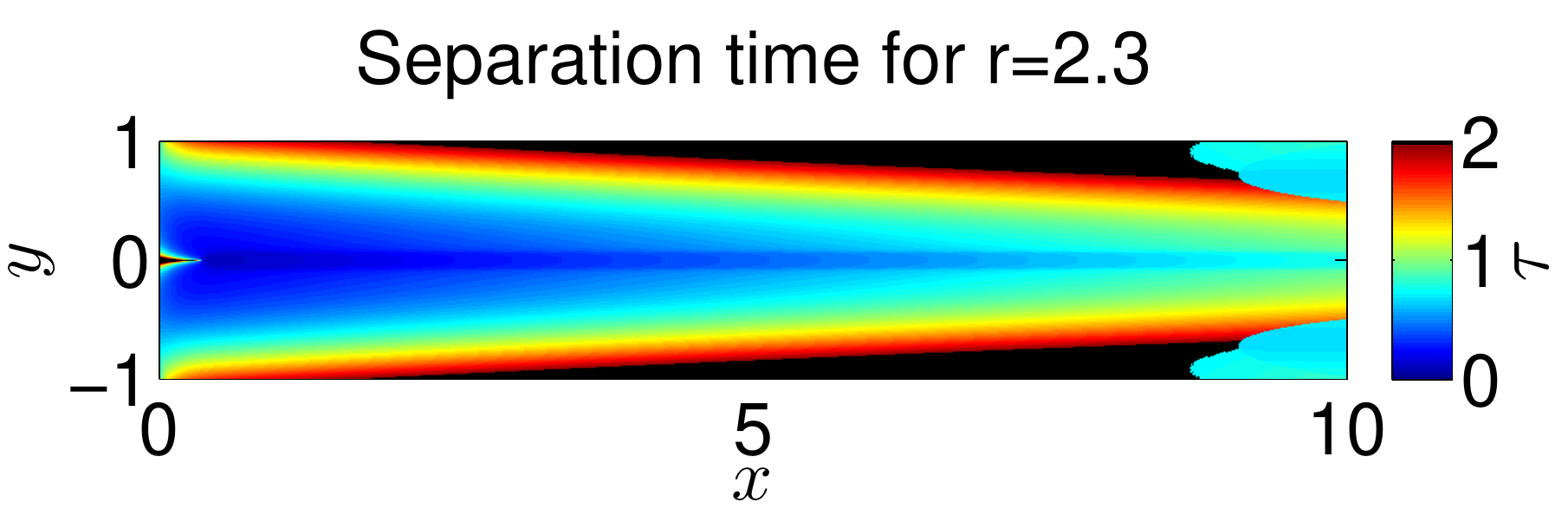}\\
\includegraphics[width=0.49\textwidth]{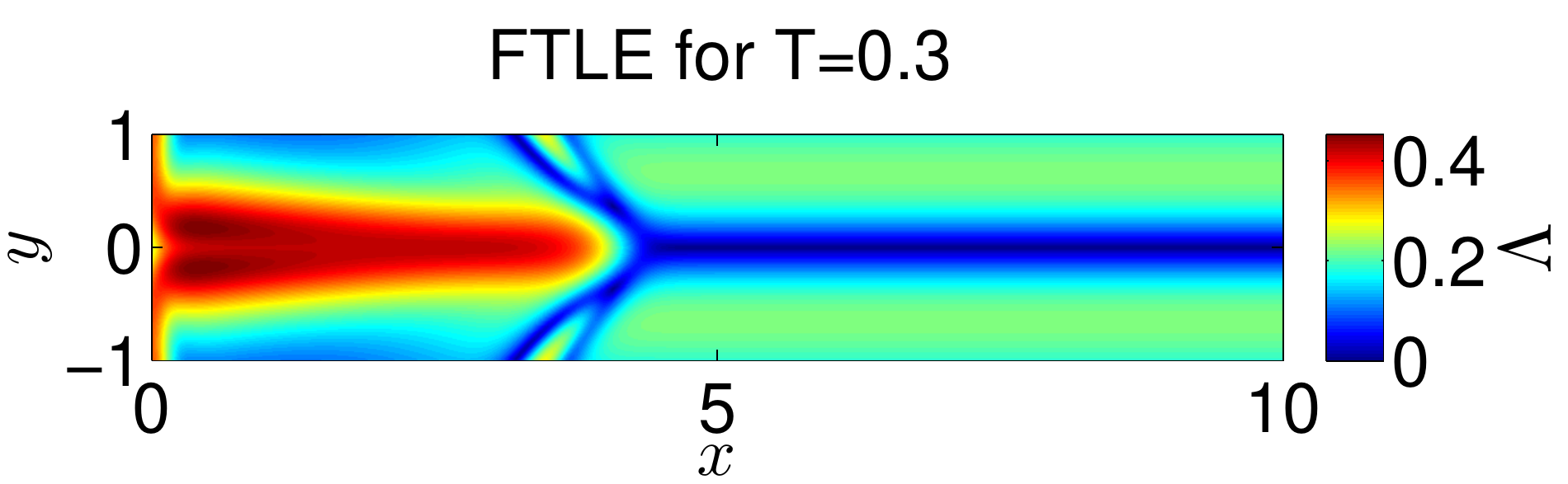}
\includegraphics[width=0.49\textwidth]{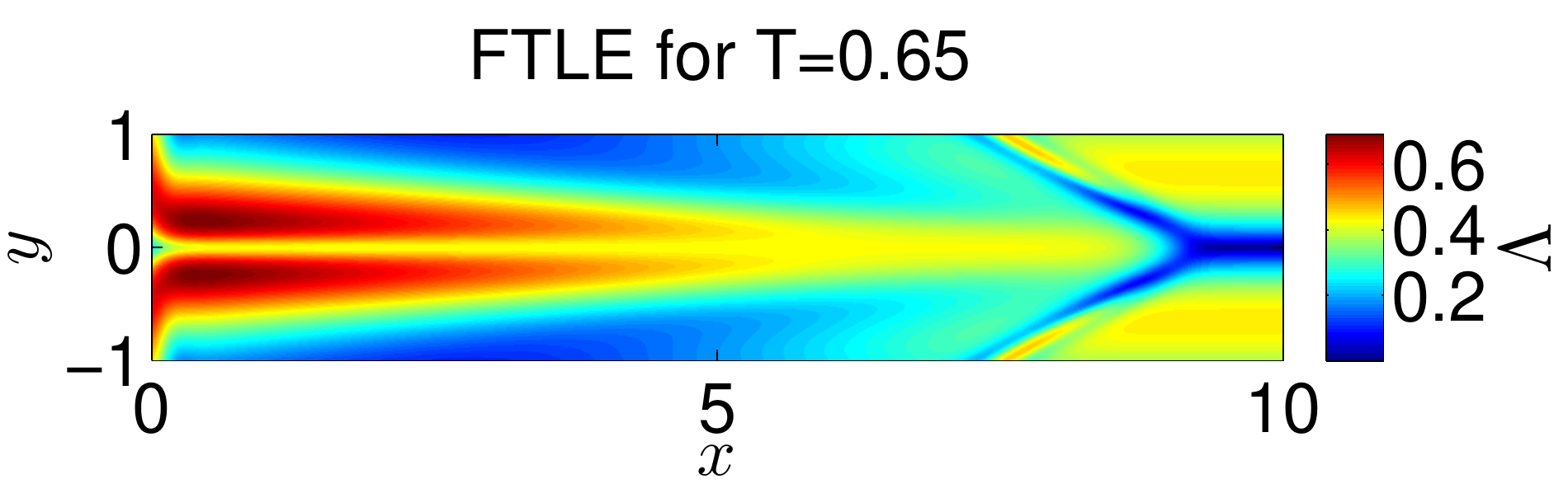}
\caption{An FSLE ridge with large height-variation that does not correspond
to an FTLE ridge. Top: FSLE and separation time distribution for the moving separation flow
\eqref{eq:separationexample} with separation factor $r=2.3$. Note
that the whole of the $x_{1}$ axis is a ridge for the FSLE field
(left), as confirmed by the separation time field (right). Bottom: The FTLE fields for the integration times $T=0.3$ and $T=0.65$ admit a trench along most of the $x_{1}$ axis.}
\label{fig:snow}
\end{figure}

\section{The Infinitesimal-Size Lyapunov Exponent (ISLE)\label{sec:ISLE}}

The examples of Section \ref{sec:FSLE-FTLE} illustrate the need to
clarify the relationship between the FSLE and FTLE fields. The first
challenge is that the FSLE is inherently linked to trajectory separation
resulting from a finite-size initial perturbation $\delta_{0}$ to
the initial condition $x_{0}$. By contrast, the FTLE describes the
separation of trajectories starting infinitesimally close to $x_{0}$.
To close this conceptual gap between the two quantities, we define
the infinitesimal analog of FSLE by taking the $\delta_{0}\to0$ limit
in its definition.

\begin{defn}
We define the \emph{Infinitesimal-Size Lyapunov Exponent (ISLE)} as
\[
\sigma_{0}\left(x_{0},r\right):=\lim_{\delta_{0}\to0}\sigma\left(x_{0};\delta_{0},r\right).
\]
 For the FSLE field to provide a meaningful measure of trajectory
separation at $x_{0}$, the ISLE field must be well-defined at $x_{0}$,
i.e., its defining limit must exist. This is a prerequisite (albeit
no guarantee) for the FSLE to detect hyperbolic LCSs reliably.

We now present a result on the existence, computation and relevance
of the limit defining $\sigma_{0}.$ In formulating these results,
we will use the infinitesimal analog $\tau_{0}(x_{0},r)$ of the finite-size
separation time $\tau(x_{0};\delta_{0},r),$ defined as
\[
\tau_{0}(x_{0},r):=\min\left\{ \left|t-t_{0}\right|:\,\, t>t_{0},\,\,\lambda_{\max}\left(C{}_{t_{0}}^{t}(x_{0})\right)=r^{2}\right\} .
\]
\end{defn}

\begin{thm}[Relation of ISLE to FSLE]
\label{thm:well-behavior} Assume that $\lambda_{\max}\left(C_{t_{0}}^{t_{0}+\tau_{0}(x_{0},r)}(x_{0})\right)$
is a simple eigenvalue and
\begin{equation}
\partial_{t}\lambda_{\max}\left(C_{t_{0}}^{t_{0}+\tau_{0}(x_{0},r)}(x_{0})\right)\neq0.\label{eq:nondeg}
\end{equation}
Then the following hold:\renewcommand{\labelenumi}{(\roman{enumi})}
\begin{enumerate}
\item The ISLE field $\sigma_{0}\left(x_{0},r\right)$ is well-defined and
$C^{2}$ at the point $x_{0}$, and can be computed as
\begin{equation}
\sigma_{0}\left(x_{0},r\right)=\Lambda_{t_{0}}^{t_{0}+\tau_{0}\left(x_{0},r\right)}\left(x_{0}\right)=\frac{\log r}{\tau_{0}\left(x_{0},r\right)},\label{eq:sigma0}
\end{equation}
with $\Lambda_{t_{0}}^{t}$ denoting the FTLE field defined in Eq.
\eqref{eq:FTLE}.
\item The FSLE field $\sigma\left(x_{0};\delta_{0},r\right)$ is also well-defined
and $C^{2}$ at the point $x_{0}$, and satisfies
\begin{equation}
\sigma\left(x_{0};\delta_{0},r\right)=\sigma_{0}\left(x_{0},r\right)+\mathcal{O}\left(\delta_{0}\right).\label{eq:FSLE-ISLE}
\end{equation}

\end{enumerate}
\end{thm}

\begin{proof}
See Appendix \ref{sec:appendix_a}.
\end{proof}

\begin{rem}
\prettyref{thm:well-behavior} shows that computing the ISLE field,
wherever it is well-defined, gives a close and smooth approximation
to the FSLE field in the same domain. The advantage of the ISLE is
that it is a pointwise indicator of finite-scale deformation, independent
of the choice of initial grid size. This makes the ISLE field amenable
to further mathematical analysis.
\end{rem}

\begin{rem}
\label{rmk:strainformula} As we show in Appendix A, condition \eqref{eq:nondeg}
can also be written in the equivalent form
\begin{equation}
\left\langle DF_{t_{0}}^{t}(x_{0})e_{\max}\left(C_{t_{0}}^{t}(x_{0})\right),S\left(F_{t_{0}}^{t}(x_{0}),t\right)DF_{t_{0}}^{t}(x_{0})e_{\max}\left(C_{t_{0}}^{t}(x_{0})\right)\right\rangle \vert_{t=t_{0}+\tau_{0}(x_{0},r)}\neq0,\label{eq:nondeg2}
\end{equation}
where $DF_{t_{0}}^{t}(x_{0})$ denotes the flow gradient, and $S(x,t)=\frac{1}{2}\left[\nabla v(x,t)+\left(\nabla v(x,t)\right)^{T}\right]$
is the Eulerian rate-of-strain tensor. Formula \eqref{eq:nondeg2}
reveals that the ISLE and FSLE fields are well-defined and smooth
at initial locations $x_{0}$ where the direction of largest Lagrangian
strain is \emph{not} mapped by the linearized flow map into a direction
of zero instantaneous Eulerian strain.

The non-degeneracy condition \eqref{eq:nondeg} (or, equivalently,
\eqref{eq:nondeg2}) will fail along codimension-one surfaces of initial
conditions in the phase space. The following proposition spells this
fact out in more precise terms.\end{rem}
\begin{prop}[Degeneracy of FSLE along hypersurfaces]
\label{prop:degeneracy}The non-degeneracy condition \eqref{eq:nondeg}
for the well-posedness of the FSLE field is generically violated along
families of $(n-1)$-dimensional hypersurfaces in the flow domain
$D\subset\mathbb{R}^{n}$. These hypersurfaces satisfy
\begin{equation}
\partial_{t}\lambda_{\max}\left(C_{t_{0}}^{t}(x_{0})\right)=0,\qquad\partial_{t}^{2}\lambda_{\max}\left(C_{t_{0}}^{t}(x_{0})\right)\neq0,\qquad\partial_{x_{0}}\lambda_{\max}\left(C{}_{t_{0}}^{t}(x_{0})\right)\neq0,\label{eq:degsurf}
\end{equation}
 with the times $t=t_{0}+\tau_{0}(x_{0},r)$ substituted after the
differentiations in \eqref{eq:degsurf}.\end{prop}
\begin{proof}
See Appendix \ref{sec:appendix_a}.
\end{proof}

\begin{rem}
\label{rmk:sensitivity}The hypersurfaces defined by formula \eqref{eq:degsurf}
define locations of jump-dis\-con\-ti\-nu\-ities for the FSLE field. In a
neighborhood of these surfaces, the FSLE will show sensitive dependence
on the numerical step-size used in its computation. This generalizes
our observations made in \prettyref{sec:sensitivity} from a specific
two-dimensional, steady flow model to unsteady flows of arbitrary
dimension. The resulting sensitivity with respect to the temporal
resolution of the data will be particularly pronounced near hyperbolic
LCSs, where the separation time $\tau(x_{0};\delta_{0},r)$ is low,
and hence errors in its computation will cause significant errors
in the FSLE. \end{rem}
\begin{example}
In the double-gyre flow \eqref{eq:gyre-1}, the jump condition \eqref{eq:degsurf}
holds along one-di\-men\-sion\-al curves, as is apparent from Fig.\ \ref{fig:single_gyre-1}.
As a consequence, the ISLE field is not smooth along these locations,
which results in jump-discontinuities in the FSLE field along crossing
curves (cf.\ Fig.\ \ref{fig:single_gyre}).
\end{example}

\begin{figure}
\centering
\includegraphics[width=0.4\textwidth]{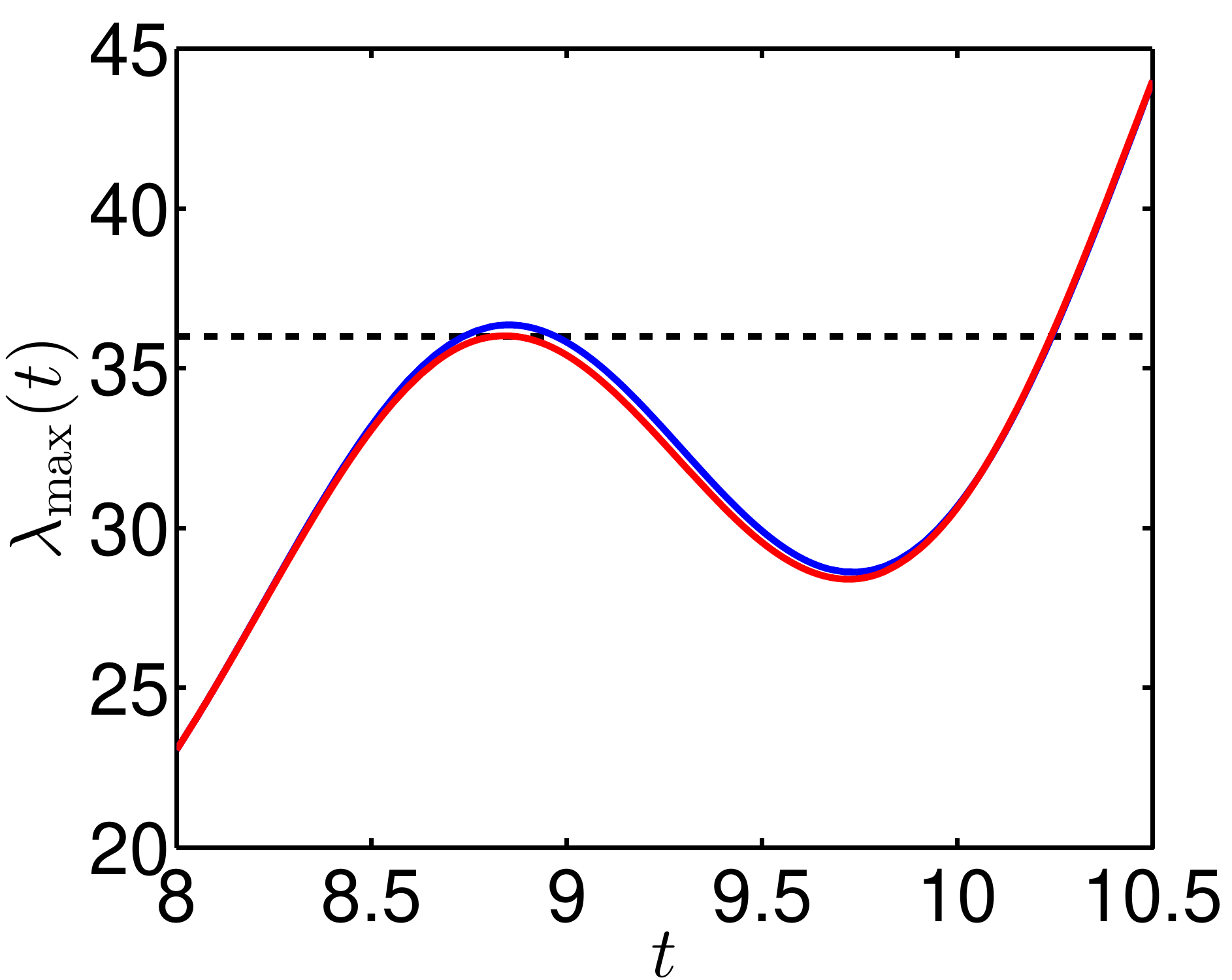}
\caption{An occurrence of the first two FSLE jump-conditions in \eqref{eq:degsurf}
along the $x_{2,\textnormal{spec}}=0.48$ line in the double-gyre
flow \eqref{eq:gyre-1}. The third jump-condition in \eqref{eq:degsurf}
is also satisfied, as the FTLE field does not have a stationary point
at this location.}
\label{fig:single_gyre-1}
\end{figure}

\section{Ridges as invariant manifolds under the gradient flow\label{sec:Ridges}}

Ridges of the FTLE field are expected to signal hyperbolic LCSs, as
initially proposed in \cite{Haller2001} (see also \cite{Shadden2011}
and \cite{Peacock2013}). While this expectation
is often justified, more recent work has revealed that FTLE ridges
may also produce both false negatives and false positives in LCS detection
\cite{Haller2011}. False positives can be filtered out by verifying
further conditions along FTLE ridges \cite{Haller2011,Haller2000-1,Farazmand2012,Karrasch2012}.
False negatives can be avoided by using more advanced, variational
LCS methods that do not rely on FTLE ridges. These methods are supported
by theorems, and render LCSs in a parametrized form, as solutions
of differential equations \cite{Haller2012,Farazmand2012-1,Farazmand2013}.

By analogy with FTLE ridges, FSLE ridges have also been assumed to
signal hyperbolic LCSs \cite{Joseph2002,dOvidio2004,Lehahn2007,Bettencourt2013,Lai2011}.
In view of the differences between the FSLE and FTLE surveyed in Section
\ref{sec:FSLE-FTLE}, a strict analogy between the ridges of the two
fields \emph{cannot} hold. Below we establish conditions under which
an FSLE ridge does signal a nearby FTLE ridge, on which further hyperbolicity
tests can be performed to ascertain the existence of a hyperbolic
LCS in the flow.

Various ridge definitions are used in topology and visualization (cf.
\cite{Eberly1994} for a general survey and \cite{Schindler2012}
for an LCS-related review). Here we introduce a new definition that
is particularly well-suited for ridge-continuation from one scalar
field to another. Specifically, we view ridges of a scalar function
$f(x)$ as codimension-one attracting invariant manifolds of the gradient
dynamical system associated with $f(x)$. The following definition
formalizes this view in mathematical terms, motivated by the FTLE-ridge
extraction technique devised by \cite{Mathur2007}.
\begin{defn}[Ridge]\label{def:ridge}
Let $f\colon\R^{n}\to\R$ be a class $C^{p}$ function
with $p\geq2$. Let $\mathcal{M}\subset\R^{n}$ be a compact, codimension-one
manifold, whose boundary $\partial\mathcal{M}$ is a compact, codimension-two
manifold without boundary. We call $\mathcal{M}$ a \emph{ridge} for
the function $f$ if both $\mathcal{M}$ and $\partial\mathcal{M}$
are normally attracting invariant manifolds for the gradient dynamical
system
\begin{equation}
\dot{x}=\nabla f(x).\label{eq:gradsystem}
\end{equation}

\end{defn}
By invariance of a manifold, we mean that trajectories of \eqref{eq:gradsystem}
starting in the manifold never leave it in either time direction.
This implies that the gradient vector field $\nabla f(x)$
is contained in the tangent space $T_{x}\mathcal{M}$ at each point
$x\in\mathcal{M}$. In addition, at each $x\in\partial\mathcal{M}$,
we must also have $\nabla f(x)\in T_{x}\partial\mathcal{M}$.

By normal attraction for $\mathcal{M}$, we mean that contraction
rates normal to $\mathcal{M}$ dominate any possible contraction rates
along $\mathcal{M}$ \cite{Fenichel1971}. Normal attraction for $\mathcal{\partial M}$
represents the same requirement, also implying that any contraction
rate within $\mathcal{\partial M}$ must be weaker than contraction
rates within $\mathcal{M}$ normal to its boundary $\mathcal{\partial M}$.

Sketched in the left plot of Fig.\ \ref{fig:manifold}, the manifold
forming the ridge $\mathcal{M}$ is robust under small perturbations
to $f$ in the following sense.
\begin{prop}[Persistence of ridges]
\label{prop:ridgepers0} A ridge in the sense of Definition \ref{def:ridge}
perturbs smoothly into a nearby $C^{1}$-close ridge under a small
enough $C^{1}$ perturbation to the function $f(x).$
\begin{proof}
See Appendix \ref{sec:appendix_b}.
\end{proof}
\end{prop}

\begin{figure}
\centering
\includegraphics[width=0.4\textwidth]{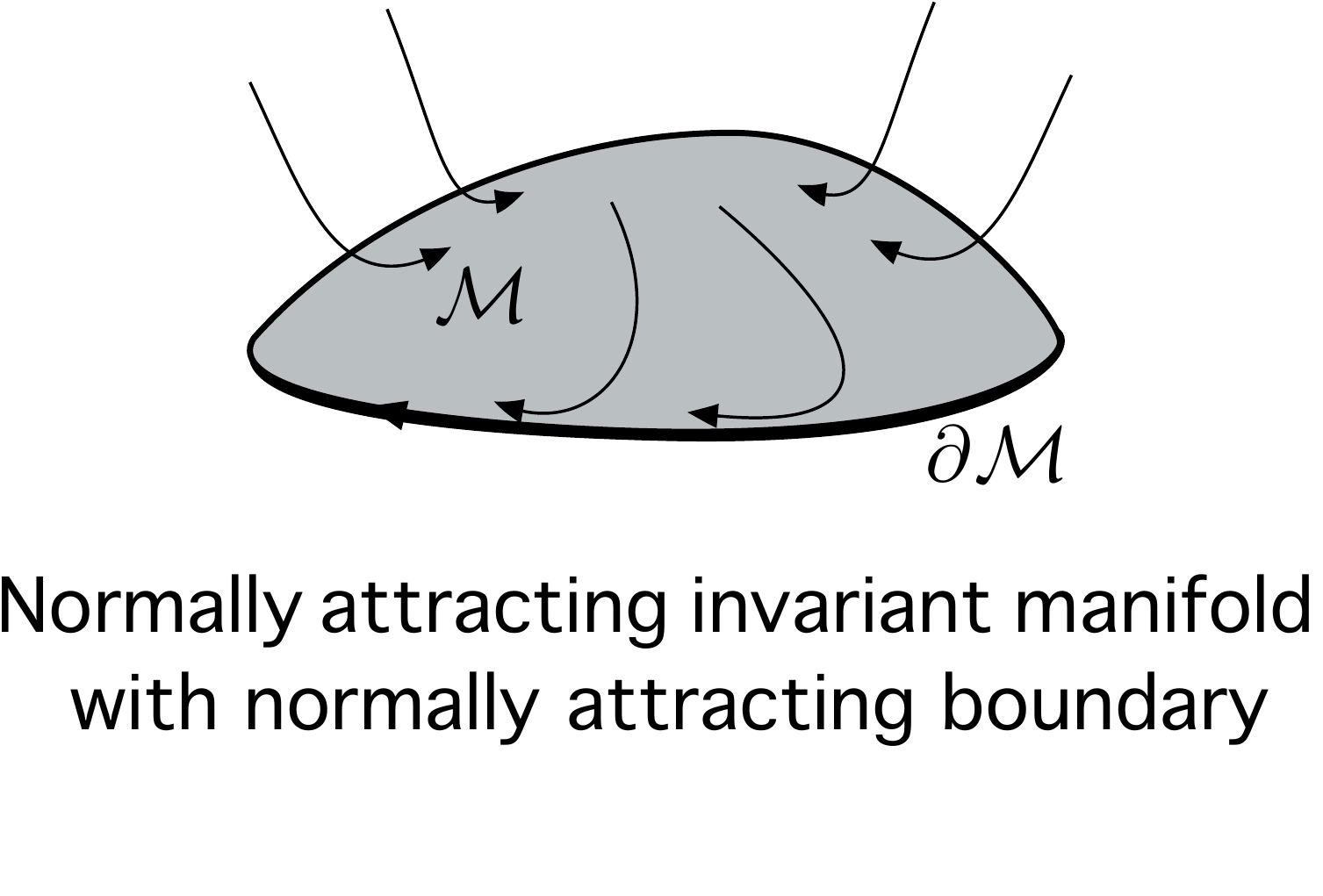}\qquad
\includegraphics[width=0.45\textwidth]{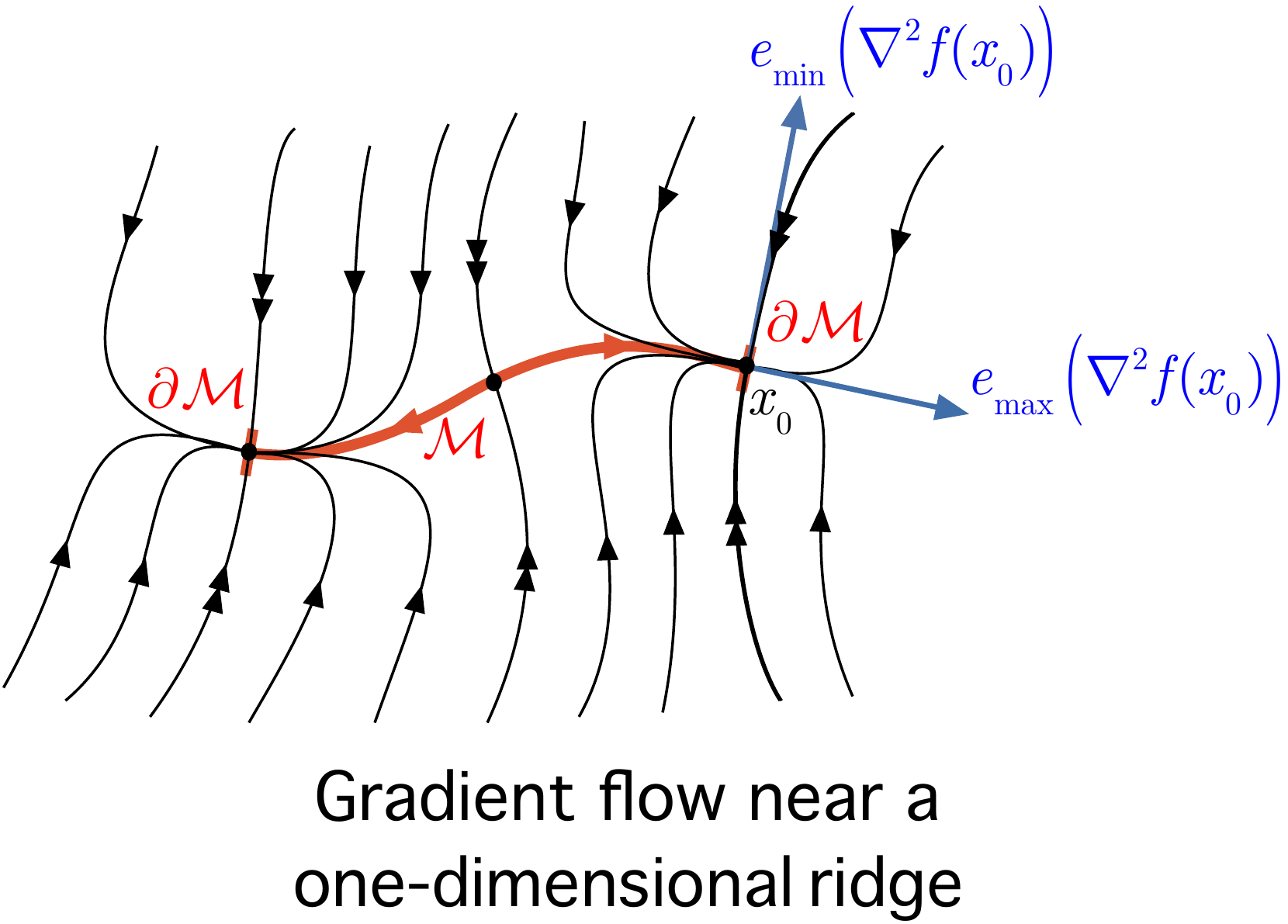}
\caption{Left: Codimension-one, normally attracting invariant manifold with
a normally attracting boundary for the gradient flow \eqref{eq:gradsystem}.
Right: The geometry of the gradient flow \eqref{eq:gradsystem} near
a one-dimensional ridge $\mathcal{M}$.}
\label{fig:manifold}
\end{figure}

The advantage of \prettyref{def:ridge} is that it guarantees robustness
for the ridge based on powerful persistence results for normally hyperbolic
invariant manifolds (see Proposition \ref{prop:ridgepers0}). Other
available ridge definitions do not provide a well-defined set of conditions
for ridge-persistence under changes in the underlying scalar field;
only partial results exist for specific cases \cite{Damon1999,Norgard2013}.

On the other hand, verifying Definition \ref{def:ridge} on a ridge-candidate
set $\mathcal{M}$ involves the computation of Lyapunov-type numbers
that guarantee normal hyperbolicity \cite{Fenichel1971}. Computing
these numbers can be laborious, requiring the numerical solution of
trajectories of the gradient system \eqref{eq:gradsystem} in $\mathcal{M}$.

This computational complexity is absent in the frequent case when
all forward-time limit sets for trajectories of \eqref{eq:gradsystem}
in $\mathcal{M}$ are fixed points. In that case, it is sufficient
to verify that the normal attraction rate to $\mathcal{M}$ and to
$\mathcal{\partial M}$ at each fixed point dominates any potential
tangential contraction rate within these manifolds at the fixed point.
This is because Lyapunov-type numbers associated with a trajectory
coincide with the Lyapunov-type numbers computed on the limit set
of the trajectory \cite{Fenichel1971}.

In the case of a one-dimensional ridge-candidate $\mathcal{M}$, all
limit sets of system \eqref{eq:gradsystem} within $\mathcal{M}$
are necessarily fixed points (see Fig.\ \ref{fig:manifold}, right).
Then, we obtain the following readily verifiable ridge criterion that
does not require the numerical solution of the gradient system \eqref{eq:gradsystem}.
\begin{prop}[Existence of one-dimensional ridges]
\label{prop:2Dridge} Let $f\colon\R^{2}\to\R$ be a class $C^{2}$
function, and let $\mathcal{M}\subset\R^{2}$ be a compact curve with
boundary. Assume that\renewcommand{\labelenumi}{(\roman{enumi})}
\begin{enumerate}
\item $\nabla f(x)\in T_{x}\mathcal{M}$ for all $x\in\mathcal{M}$;
\item $\nabla f(x)=0$ and $\lambda_{\mathrm{max}}\left(\nabla^{2}f(x)\right)<0$
for both points $x\in\mathcal{\partial M}$.
\item For all points $x_{0}\in\mathcal{M}$, where $\nabla f(x_{0})=0$,
the Hessian $\nabla^{2}f(x_{0})$ has simple eigenvalues, with the
smaller eigenvalue satisfying
\begin{align*}
\lambda_{\min}\left(\nabla^{2}f(x_{0})\right) & <0, & e_{\min}\left(\nabla^{2}f(x_{0})\right) & \perp T_{x_{0}}\mathcal{M}.
\end{align*}

\end{enumerate}
Then $\mathcal{M}$ is a ridge for the function $f$ in the sense
of Definition \ref{def:ridge}.\end{prop}
\begin{proof}
See Appendix \ref{sec:appendix_b}.\end{proof}
\begin{rem}
As seen in Fig.\ \ref{fig:manifold} and as required by condition (2)
of Proposition \ref{prop:2Dridge}, a one-dimensional ridge in the
sense of Definition \ref{def:ridge} is necessarily a curve connecting
two local maxima of a scalar field $f(x)$, and containing at least
one more critical point of $\nabla f(x)$ in its interior.
\end{rem}

\begin{rem}
As noted by \cite{Schindler2012}, requiring one of the eigenvectors of the Hessian $\nabla^{2}f(x)$
of a scalar field $f(x)$ to be parallel to a ridge $\mathcal{M}$ at all points $x\in\mathcal{M}$
leads to an over-constrained ridge definition. Our \prettyref{def:ridge} implies that one of the eigenvectors of $\nabla^{2}f(x_{0})$ is automatically parallel to $\mathcal{M}$
at any critical point $x_{0}$ of $f(x)$. This follows from the fact that
$\mathcal{M}$ is an invariant manifold for the gradient flow $\dot{x}=\nabla f(x)$,
and hence $T_{x_{0}}\mathcal{M}$ is necessarily an invariant subspace for the linearized gradient flow $\dot{y}=\nabla^{2}f(x_{0})y$ at any critical point $x_{0}\in\mathcal{M}$
of the function $f(x)$. Condition (iii) of \prettyref{prop:2Dridge} simply adds the requirement that the ridge-parallel eigenvector at $x_{0}$ should be the eigenvector corresponding to the smaller eigenvalue of the Hessian $\nabla^{2}f(x_{0})$.
\end{rem}

\section{When does an FSLE ridge signal a nearby FTLE ridge?}

The following result establishes that an FSLE ridge indicates a nearby
FTLE ridge, provided that the initial separation distance $\delta_{0}$
is small enough, and the ISLE separation times $\tau_{0}\left(x;r\right)$
along the FSLE ridge are close enough to a constant value in the $C^{2}$
norm.
\begin{thm}[Continuation of FSLE ridges into FTLE ridges]
\label{thm:ridge2ridge}  Let $\cM$ be a ridge of the FSLE field
$\sigma\left(x;\delta_{0},r\right)$ in the sense of Definition \ref{def:ridge}.
Assume that in a compact neighborhood $U$ of $\mathcal{M}$, we have
\begin{align}
\partial_{t}\lambda_{\max}\left(C_{t_{0}}^{t_{0}+\tau_{0}(x,r)}(x)\right) &\neq 0, & \left\Vert \tau_{0}\left(x;r\right)-\bar{\tau}_{0}\right\Vert _{C^{2}} &\leq\varepsilon, & x &\in U,\label{eq:C2cond}
\end{align}
 for appropriate constants $\bar{\tau}_{0}>0$ and $0\leq\varepsilon$,
and with $\left\Vert \cdot\right\Vert _{C^{2}}$ referring to the
$C^{2}$ norm. Then, for $\varepsilon,\delta_{0}$ sufficiently small,
the FTLE field $\Lambda_{t_{0}}^{t_{0}+\bar{\tau}_{0}}\left(x\right)$
has a ridge $\mathcal{\bar{M}}$ that is $\mathcal{O}(\varepsilon,\delta_{0})$
~$C^{1}$-close to $\cM$. \end{thm}
\begin{proof}
See Appendix \ref{sec:appendix_b}.\end{proof}
\begin{rem}

\prettyref{thm:ridge2ridge} implies that there is an open set of
$\bar{\tau}_{0}$ values for which the $\Lambda_{t_{0}}^{t_{0}+\bar{\tau}_{0}}\left(x\right)$
field will admit a nearby ridge. Indeed, small enough changes in the
constant $\bar{\tau}_{0}$ will not affect the statement of the theorem.
\end{rem}

\begin{rem}
By formula \eqref{eq:sigma0}, the second condition in \eqref{eq:C2cond}
is equivalent to
\begin{align*}
\left\Vert \frac{1}{\bar{\sigma}}-\frac{1}{\sigma\left(x_{0},r\right)}\right\Vert _{C^{2}} &\leq\frac{\varepsilon}{\log r}, & \bar{\sigma} & \coloneqq\frac{\log r}{\bar{\tau}_{0}}.
\end{align*}
Therefore, one may equivalently require small enough variations in
the reciprocal of ISLE field $\sigma\left(x_{0},r\right)$ in the
$C^{2}$ norm within a compact neighborhood $U$ of the ridge $\mathcal{M}$.
This in turn can be enforced by requiring small enough variations
in the FSLE field along $\mathcal{M}$ by formula \eqref{eq:FSLE-ISLE}.
\end{rem}

\begin{rem}
By Fenichel's results \cite{Fenichel1971}, the FTLE ridge is only
guaranteed to be $\mathcal{O}(\varepsilon,\delta_{0})$ $C^{1}$-close
to the original FSLE ridge. This means that the two ridges are pointwise
close and their tangent spaces at these points are also close. Closeness
of the curvatures of the two ridges, however, does not immediately
follow in our setting.
\end{rem}

\begin{rem}\label{rmk:peaks}
Assume that the FSLE field is of class $C^{s}$ with $s\geq 2$, and its ridge is a $C^{p}$
differentiable manifold with $p\geq 2$. Then the maximum degree of smoothness guaranteed for a nearby FTLE ridge will be
\[
q=\min\left(s-1,p,\min_{x_{0}\in Z_{0}}\mathrm{Int}\left[\frac{\lambda_{\mathrm{min}}\left(\nabla^{2}f(x_{0})\right)}{\lambda_{\max}\left(\nabla^{2}f(x_{0})\right)}\right]\right)
\]
by the theory of normally hyperbolic invariant manifolds \cite{Fenichel1971}. Here we have used the set
\[
Z_{0}=\left\{ x_{0}\in\mathcal{M}:\,\,\nabla f(x_{0})=0,\,\,\lambda_{\mathrm{min}}\left(\nabla^{2}f(x_{0})\right)<\lambda_{\max}\left(\nabla^{2}f(x_{0})\right)<0\right\},
\]
as well as the notation $\mathrm{Int\left[\cdot\right]}$ for the integer part of a positive real number. The quotient $\lambda_{\mathrm{min}}\left(\nabla^{2}f(x_{0})\right)/\lambda_{\max}\left(\nabla^{2}f(x_{0})\right)$ is just the Lyapunov-type number introduced by \cite{Fenichel1971}, computed at stationary values $x_{0}$ of $f$ along $\mathcal{M}$. The minimum of these Lyapunov-type numbers potentially limits the differentiability of the nearby FTLE ridge further, as seen from the formula defining $q$. In general, the larger the minimal Lyapunov-type number along the FSLE ridge, the more robust the ridge is under perturbations, i.e., the larger $\delta_{0}$ and $\epsilon$ can be selected in the statement of \prettyref{thm:ridge2ridge}.
\end{rem}

\begin{example}
\label{ex:ridge_nocont} The moving separation example in
\prettyref{sec:movesep} shows that large variations in the height
of FSLE ridges do indeed result in the non-persistence of these ridges
in the FTLE field. In this example, a large variation in $\tau_{0}$
is observed along the ISLE ridge defined by $x_{2}=0$, see Fig.\ \ref{fig:snow-1}.
As a result, no constant $\bar{\tau}_{0}$ satisfying \eqref{eq:C2cond}
can be selected for small values of $\varepsilon>0$.
\end{example}

\begin{figure}
\centering
\includegraphics[width=0.5\textwidth]{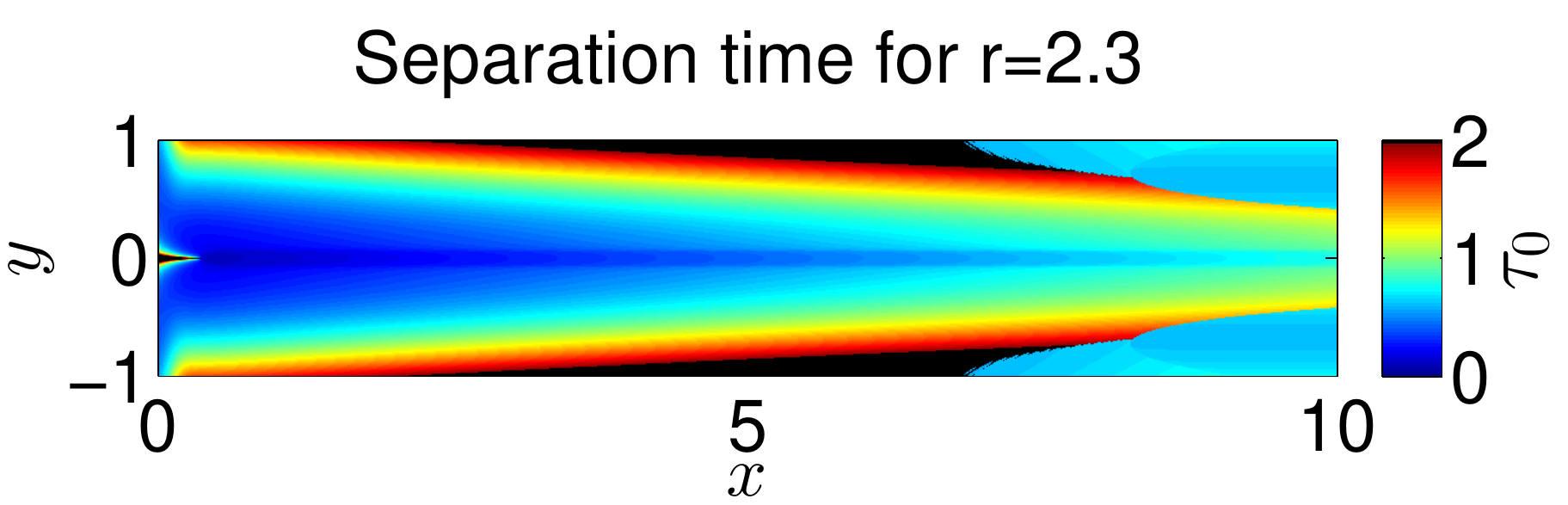}
\caption{The separation time distribution for the moving separation flow \prettyref{eq:separationexample}
with separation factor $r=2.3$.}
\label{fig:snow-1}
\end{figure}

\prettyref{thm:ridge2ridge} and Remark \ref{rmk:peaks} show that
FSLE ridges with small enough variations in their ISLE values in the
$C^{2}$ norm, and with large enough transverse steepness at their
peaks, give rise to nearby FTLE ridges. Example \ref{ex:ridge_nocont}
shows that in flows violating this requirement, either no or several
$C^{1}$-close FTLE ridges may exist. Therefore, the types of conditions
required in \prettyref{thm:ridge2ridge} are indeed necessary for FSLE
ridges to be meaningful in hyperbolic LCS detection, even though the
constants arising in these conditions are not readily computable from our proof.

\section{Inferring hyperbolic LCS from FSLE ridges\label{sec:route}}

While select FSLE ridges signal the presence of nearby FTLE ridges
by Theorem \ref{thm:ridge2ridge}, this does not imply that there
is always a corresponding hyperbolic LCS in the flow. Indeed, simple
examples show that an FTLE ridge may simply indicate locations of
locally maximal shear \cite{Haller2002,Haller2011}.

More recent variational methods enable the direct extraction of hyperbolic
LCSs as pa\-ram\-e\-trized curves \cite{Farazmand2012-1}. Further
generalizations extend this computational advantage to parabolic and
elliptic LCSs as well \cite{Haller2012,Haller2013a}.

These high-end detection techniques also require additional computational
investment that ensures the accurate solution of differential equations
derived from the eigenvector fields of the Cauchy--Green strain tensor.
For a rough first assessment of hyperbolic LCSs, one may simply check
additional criteria along FTLE ridges to conclude the existence of
nearby hyperbolic LCSs. We refer the reader to \cite{Haller2011,Haller2000-1,Farazmand2012,Karrasch2012}
for such criteria.

Conversely, given a spatial scale of interest, a preliminary FSLE
analysis might be helpful in determining a relevant time scale of
integration to be used in variational LCS methods.

\section{Conclusions\label{sec:conclusions}}

Using the Infinitesimal-Size Lyapunov Exponent (ISLE), we have established
a link between certain ridges of the FSLE field and those of the FTLE
field. Specifically, FSLE ridges with moderate ISLE variations and
high normal steepness at their peaks signal nearby FTLE ridges, as
long as the time-derivative of the largest eigenvalue of the Cauchy-Green
strain tensor is nonzero in a neighborhood of these FSLE ridges (cf.
Theorem \ref{thm:ridge2ridge} and Remark \ref{rmk:peaks}). This
nonzero derivative condition will, however, be violated along families
of hypersurfaces in the phase space, over which the FSLE field admits
jump-discontinuities.

Families of such FSLE jump-surfaces are generically present in any
nonlinear flow (Proposition \ref{prop:degeneracy}), creating sensitivity
in FSLE computations with respect to the temporal resolution of the
underlying flow data (cf.\ Remark \ref{rmk:sensitivity}). This sensitivity
may even impact the accuracy of FSLE statistics, as has already been
observed for float experiments in the ocean \cite{LaCasce2008}.

In addition to jump-discontinuities and the associated temporal sensitivity,
we have also identified further disadvantages of the FSLE field in
detecting Lagrangian coherence. These include ill-posedness for ranges
of the separation parameter $r$, insensitivity to changes in the
flow once the separation time $\tau$ is reached, and non-existence
of nearby FTLE ridges in the case of FSLE ridges with substantial
variation in their heights. We have illustrated all these issues with
the FSLE in simple examples.

These findings suggest that the simplicity of computing FSLE comes
at a price. If the objective is the accurate and threshold-free detection
of hyperbolic LCSs, then more recent variational LCS techniques offer
multiple advantages over FSLE-based coherence detection. While these
variational techniques require a higher computational investment,
they do provide a full and rigorous detection of all types of LCSs,
including hyperbolic, parabolic, and elliptic LCSs \cite{Haller2012,Beron-Vera2013,Haller2013a}.
This is to be contrasted with the substantial cost of varying two
free parameters and with the remaining uncertainty in the results,
if LCSs are to be inferred from the FSLE field without further mathematical
analysis. On the upside, given a spatial scale of interest, FSLE may
help in identifying the integration times to be used in variational
LCS methods.

In summary, the use of the FSLE in hyperbolic LCS detection requires
caution. Only flows with high temporal resolution and limited unsteadiness
can be reliably analyzed. In addition, only ridges with moderate variations
in their height and with high enough normal steepness at their peaks
can be guaranteed to signal nearby LCSs. Even in such flows, the FSLE
field will show sensitivity near hypersurfaces defined by the equation
$\partial_{t}\lambda_{\max}\left(C_{t_{0}}^{t_{0}+\tau_{0}(x,r)}(x)\right)=0.$
This sensitivity is the highest near hyperbolic LCSs, as these lead
to low values of the separation time, whose reciprocal values magnify
errors in the FSLE field.

\subsection*{Acknowledgements}

We thank Ronny Peikert for suggesting references on available persistence
results for ridges. We also thank Joe LaCasce for stimulating discussions
on Lagrangian statistics and their relation to the FSLE. We are grateful
to Dan Blazevski and Mohammad Farazmand for their insights into the
definition of ridges.

\appendix

\section{Proof of Theorem \ref{thm:well-behavior}, Remark \ref{rmk:strainformula}
and Proposition \ref{prop:degeneracy}}\label{sec:appendix_a}

\subsection{Proof of Theorem \ref{thm:well-behavior}}

The separation time $\tau(x_{0};\delta_{0},r)$ at the initial condition
$x_{0}$ is the smallest positive solution of the equation
\begin{equation}
\begin{split}r^{2}\delta_{0}^{2} & =\left\vert F_{t_{0}}^{t_{0}+\tau}(y_{0})-F_{t_{0}}^{t_{0}+\tau}(x_{0})\right\vert ^{2}\\
 & =\left\vert DF_{t_{0}}^{t_{0}+\tau}(x_{0})\left(y_{0}-x_{0}\right)+\mathcal{O}\left(\left\vert y_{0}-x_{0}\right\vert ^{2}\right)\right\vert ^{2}\\
 & =\delta_{0}^{2}\left\langle e\left(x_{0}\right),\left[DF_{t_{0}}^{t_{0}+\tau}(x_{0})\right]^{T}DF_{t_{0}}^{t_{0}+\tau}(x_{0})e\left(x_{0}\right)\right\rangle +\mathcal{O}\left(\delta_{0}^{3}\right)\\
 & =\delta_{0}^{2}\left\langle e\left(x_{0}\right),C_{t_{0}}^{t_{0}+\tau}(x_{0})e\left(x_{0}\right)\right\rangle +\mathcal{O}\left(\delta_{0}^{3}\right),
\end{split}
\label{eq:disteq1}
\end{equation}
where
\[
e\left(x_{0}\right)=\frac{y_{0}-x_{0}}{\left\vert y_{0}-x_{0}\right\vert },
\]
is the unit vector pointing from $x_{0}$ towards $y_{0}$. Dividing
\eqref{eq:disteq1} by $\delta_{0}^{2}$, we obtain
\begin{equation}
\left\langle e\left(x_{0}\right),C_{t_{0}}^{t_{0}+\tau}(x_{0})e\left(x_{0}\right)\right\rangle +\mathcal{O}\left(\delta_{0}\right)=r^{2},\label{eq:disteq2}
\end{equation}
which is equivalent to \eqref{eq:disteq1} for all $\delta_{0}>0$.

By continuity of all quantities involved in \eqref{eq:disteq2}, the
limit $\tau(x_{0},r)=\lim_{\delta_{0}\to0}\tau(x_{0};\delta_{0},r)$
must coincide with the minimal solution $\tau$ of the equation
\begin{equation}
\left\langle e\left(x_{0}\right),C_{t_{0}}^{t_{0}+\tau}(x_{0})e\left(x_{0}\right)\right\rangle =r^{2}.\label{eq:disteq3}
\end{equation}

To explore the solvability of the limiting equation \eqref{eq:disteq3},
recall that the Cauchy--Green strain tensor $C_{t_{0}}^{t_{0}+\tau}(x_{0})$
is symmetric, positive definite, and satisfies $C_{t_{0}}^{t_{0}}(x_{0})=I$,
with $I$ denoting the identity matrix. Consequently, $\tau_{0}(x_{0},r)\coloneqq\tau(x_{0};0,r)$
is the smallest positive solution of \eqref{eq:disteq3} if $e\left(x_{0}\right)$
is chosen as the unit dominant eigenvector $e_{\max}\left(C_{t_{0}}^{t_{0}+\tau_{0}(x_{0},r)}(x_{0})\right)$
of the associated Cauchy--Green strain tensor. In that case, an equivalent
equation for the smallest positive root of \eqref{eq:disteq3} is
given by
\begin{equation}
\lambda_{\max}\left(C_{t_{0}}^{t_{0}+\tau_{0}(x_{0},r)}(x_{0})\right)=r^{2},\label{eq:disteq4}
\end{equation}
which implies the formula \eqref{eq:sigma0}, and hence proves statement
(i) of Theorem \ref{thm:well-behavior} with the exception of the
claim of $C^{2}$ smoothness.

To prove statement (ii) of Theorem \ref{thm:well-behavior} and the
$C^{2}$ smoothness in statement (i), we want to continue the solution
of equation \eqref{eq:disteq2} smoothly from $\delta_{0}=0$ to $\delta_{0}>0$
values. By the implicit function theorem, this continuation requires
precisely condition \eqref{eq:nondeg} to hold. Also, the continued
solution will be $C^{2}$ smooth by the implicit function theorem,
given that the right-hand side of \eqref{eq:velo}, and hence the
flow map, are assumed to be $C^{3}$ smooth. Consequently, statement
(ii) of Theorem \ref{thm:well-behavior} follows.

\subsection{Proof of Remark \ref{rmk:strainformula}}

Observe that the derivative in the non-degeneracy condition \eqref{eq:nondeg}
can be computed as
\begin{align}
\partial_{\tau}\lambda_{\max}\left(C_{t_{0}}^{t_{0}+\tau}(x_{0})\right) & =\partial_{\tau}\left\langle e_{\mathrm{max}}\left(C_{t_{0}}^{t_{0}+\tau}(x_{0})\right),C_{t_{0}}^{t_{0}+\tau}(x_{0})e_{\mathrm{max}}\left(C_{t_{0}}^{t_{0}+\tau}(x_{0})\right)\right\rangle \nonumber \\
 & =\left\langle e_{\mathrm{max}}\left(C_{t_{0}}^{t_{0}+\tau}(x_{0})\right),\partial_{\tau}C_{t_{0}}^{t_{0}+\tau}(x_{0})e_{\mathrm{max}}\left(C_{t_{0}}^{t_{0}+\tau}(x_{0})\right)\right\rangle ,\label{eq:cgder0}
\end{align}
where we have used the fact that
\[
\partial_{\tau}e_{\mathrm{max}}\left(C_{t_{0}}^{t_{0}+\tau}(x_{0})\right)\perp e_{\mathrm{max}}\left(C_{t_{0}}^{t_{0}+\tau}(x_{0})\right),
\]
 given that $\left|e_{\mathrm{max}}\left(C_{t_{0}}^{t_{0}+\tau}(x_{0})\right)\right|\equiv1.$
Furthermore, we have
\begin{multline}
\partial_{\tau}C_{t_{0}}^{t_{0}+\tau}(x_{0})\bigr|{}_{\tau_{0}(x_{0},r)}=\left[\partial_{\tau}\left(DF_{t_{0}}^{t_{0}+\tau}(x_{0})\right)^{T}DF_{t_{0}}^{t_{0}+\tau}(x_{0})+\right.\\
\left.\left.+\left(DF_{t_{0}}^{t_{0}+\tau}(x_{0})\right)^{T}\partial_{\tau}\left(DF_{t_{0}}^{t_{0}+\tau}(x_{0})\right)\right]\right|{}_{\tau_{0}(x_{0},r)}.\label{eq:cgder}
\end{multline}
 We recall that the deformation gradient $DF_{t_{0}}^{t_{0}+\tau}(x_{0})$
satisfies the equations of variation
\begin{equation}
\partial_{\tau}\left(DF_{t_{0}}^{t_{0}+\tau}(x_{0})\right)=\partial_{x}v\left(F_{t_{0}}^{t_{0}+\tau}(x_{0}),\tau\right)DF_{t_{0}}^{t_{0}+\tau}(x_{0}).\label{eq:vari}
\end{equation}
Substituting expression \eqref{eq:vari} into \eqref{eq:cgder}, then
the resulting equation into \eqref{eq:cgder0} shows that conditions
\eqref{eq:nondeg} and \eqref{eq:nondeg2} are indeed equivalent,
as claimed in Remark \ref{rmk:strainformula}.

\subsection{Proof of Proposition \ref{prop:degeneracy}}

First, note that points violating the conditions for the well-posedness
of the FSLE satisfy the two scalar equations
\begin{eqnarray}
\lambda_{\max}\left(C{}_{t_{0}}^{t_{0}+\tau_{0}}(x_{0})\right)-r^{2} & = & 0,\label{eq:deg1}\\
\partial_{t}\lambda_{\max}\left(C_{t_{0}}^{t_{0}+\tau_{0}}(x_{0})\right) & = & 0.\label{eq:deg2}
\end{eqnarray}
 Assume that an isolated solution $(\bar{\tau}_{0},\bar{x}_{0})$
exists to this system of equations, such that $\bar{\tau}_{0}>0$
is also the minimal solution of \eqref{eq:deg1} for $\bar{x}_{0}$.
Then, by definition, $\bar{\tau}_{0}(\bar{x}_{0})$ is the separation
time for the initial condition $\bar{x}_{0}$ with separation factor
$r$. Furthermore, this separation time violates the non-degeneracy
condition \eqref{eq:nondeg}. Since being a minimal solution is an
open property and \eqref{eq:deg1} is continuous in its arguments, any
solutions $(\tau_{0},x_{0})$ of system \eqref{eq:deg1}-\eqref{eq:deg2}
close enough to $(\bar{\tau}_{0},\bar{x}_{0})$ will also define a
separation time and its corresponding location.

We would like to argue that a set of nearby solutions to equation
\eqref{eq:nondeg} generically exists and forms a smooth, $(n-1)$-dimensional
surface in the space of the $(\tau_{0},x_{0})$ variables. To this
end, we let $x_{0}^{1}$ denote the first coordinate of $x_{0}$ and
let $x_{0}^{(n-1)}$ denote the remaining $(n-1)$ coordinates of
$x_{0}$, so that $x_{0}=(x_{0}^{1},x_{0}^{(n-1)})$. We seek to establish
that near the solution $(\bar{\tau}_{0},\bar{x}_{0}^{1},\bar{x}_{0}^{(n-1)})$,
the system of equation \eqref{eq:deg1}-\eqref{eq:deg2} will continue
to admit a smooth solution of the form
\[
(\tau_{0},x_{0}^{1})=(\bar{\tau}_{0},\bar{x}_{0}^{1})+h\left(x_{0}^{(n-1)}-\bar{x}_{0}^{(n-1)}\right),
\]
where $h\colon U\subset\mathbb{R}^{n-1}\to\mathbb{R}^2$ is a smooth function with
$h(0)=0$, defined in a neighborhood $U$ of the origin $0\in\mathbb{R}^{n-1}$.
This follows from a direct application of the implicit function theorem
to the equations \eqref{eq:deg1}-\eqref{eq:deg2}, provided that
\begin{equation}\left.
\det\begin{pmatrix}\partial_{t}\lambda_{\max}\left(C{}_{t_{0}}^{t}(x_{0})\right) & \partial_{x_{0}^{1}}\lambda_{\max}\left(C{}_{t_{0}}^{t}(x_{0})\right)\\
\partial_{t}^{2}\lambda_{\max}\left(C_{t_{0}}^{t}(x_{0})\right) & \partial_{x_{0}^{1}}\partial_{t}\lambda_{\max}\left(C_{t_{0}}^{t}(x_{0})\right)
\end{pmatrix}\right|_{t=t_{0}+\bar{\tau}_{0},x_{0}=\bar{x}_{0}}\neq0.\label{eq:nondegmat}
\end{equation}

By equation \eqref{eq:deg2}, the first diagonal entry of the matrix
in \eqref{eq:nondegmat} is zero, and hence \eqref{eq:nondegmat}
is equivalent to
\[
\partial_{t}^{2}\lambda_{\max}\left(C_{t_{0}}^{t}(x_{0})\right)\partial_{x_{0}^{1}}\lambda_{\max}\left(C{}_{t_{0}}^{t}(x_{0})\right)\neq0,\qquad t=t_{0}+\bar{\tau}_{0},\quad x_{0}=\bar{x}_{0}.
\]
This latter condition is satisfied as long as (1) $\lambda_{\max}\left(C_{t_{0}}^{t}(\bar{x}_{0})\right)$
has a non-degenerate temporal maximum at the degenerate separation
time $\bar{\tau}_{0}(\bar{x}_{0})$ separation,
and (2) the maximal eigenvalue $\lambda_{\max}\left(C{}_{t_{0}}^{t}(x_{0})\right)$
varies strictly in the $x_{0}^{1}$ direction. Condition (1) holds
by the first inequality in \eqref{eq:degsurf}. Condition (2) can
always be satisfied by a possible reordering of the coordinates of
the vector $x_{0}$, given that the second inequality in \eqref{eq:degsurf}
is assumed to hold.

\section{Proof of Proposition \ref{prop:ridgepers0}, Proposition
\ref{prop:2Dridge}, and Theorem \ref{thm:ridge2ridge}}\label{sec:appendix_b}

\subsection{Proof of Proposition \ref{prop:ridgepers0}}

Elements of this flow geometry sketched in Fig.\ \ref{fig:manifold}
were studied by Fenichel \cite{Fenichel1971}, who established general
persistence results for compact, normally hyperbolic invariant manifolds
under small perturbations. These results imply that $\mathcal{\partial M}$
smoothly and uniquely persists in the form of a nearby attracting,
invariant manifold $\overline{\partial\mathcal{M}}$ under small perturbations.
Furthermore, $\mathcal{M}$ can be slightly enlarged into a normally
attracting, inflowing invariant manifold $\mathcal{N}$ beyond its
boundary. (An inflowing invariant manifold is a manifold tangent to
the underlying vector field, such that the vector field points strictly
inwards along the boundary of the manifold.) By Fenichel \cite{Fenichel1971},
such a manifold $\mathcal{N}$ also persists smoothly (but typically
not uniquely) as an attracting, inflowing invariant manifold $\overline{\mathcal{N}}$.

Now $\overline{\partial\mathcal{M}}$ necessarily lies in the domain
of attraction of $\overline{\mathcal{N}}$, which is only possible
if $\overline{\partial\mathcal{M}}\subset\overline{\mathcal{N}}$.
Then the closure of the interior of $\overline{\partial\mathcal{M}}$
within $\overline{\mathcal{N}}$, which we denote by $\overline{\mathcal{M}}$,
is a codimension-one, normally attracting invariant manifold such
that its boundary satisfies $\partial\overline{\mathcal{M}} =\overline{\partial\mathcal{M}}$. Consequently, the original manifold $\mathcal{M}$
has smoothly perturbed into $\overline{\mathcal{M}}$ under small
enough perturbations, as claimed.

\subsection{Proof of Proposition \ref{prop:2Dridge}}

Condition (1) of Proposition \ref{prop:2Dridge} ensures that the
codimension-one manifold $\mathcal{M}$ is invariant under the flow
of \eqref{eq:gradsystem}. Condition (2) ensures that the boundary
points (which are necessarily fixed points by the invariance of $\mathcal{M}$)
are attracting along $\mathcal{M}$. Condition (3) ensures that at
the fixed points of the gradient flow \eqref{eq:gradsystem} contained
in $\mathcal{M}$, the contraction rates normal to $\mathcal{M}$
dominate any possible contraction rate inside $\mathcal{M}$. Since
the asymptotic normal attracting properties of trajectories coincide
with those of their limit sets \cite{Fenichel1971}, normal attraction
for the whole of $\mathcal{M}\subset\R^{2}$ follows from the fact
that normal attraction holds at all fixed points of \eqref{eq:gradsystem}
inside $\mathcal{M}$.

\subsection{Proof of Theorem \ref{thm:ridge2ridge}}

The first condition in \eqref{eq:C2cond} ensures that both the FSLE
and ISLE fields remain well-defined and smooth in the whole compact
neighborhood $U$. Then, by the second condition in \eqref{eq:C2cond},
we can write
\begin{align*}
\sigma\left(x;\delta_{0},r\right) & =\sigma_{0}\left(x;r\right)+\mathcal{O}_{2}(\delta_{0})=\Lambda_{t_{0}}^{t_{0}+\tau_{0}(x_{0},r)}\left(x\right)+\mathcal{O}_{2}(\delta_{0})=\Lambda_{t_{0}}^{t_{0}+\bar{\tau}_{0}+\mathcal{O}(\varepsilon)}\left(x\right)+\mathcal{O}_{2}(\delta_{0})\\
 & =\Lambda_{t_{0}}^{t_{0}+\bar{\tau}_{0}}\left(x\right)+\mathcal{O}_{2}(\varepsilon,\delta_{0}),
\end{align*}
 where the $\mathcal{O}_{2}(\delta_{0})$ and $\mathcal{O}_{2}(\varepsilon,\delta_{0})$
terms denote a small, $C^{2}$ perturbation to the function $\Lambda_{t_{0}}^{t_{0}+\bar{\tau}_{0}}\left(x\right)$.
As a result, we have
\begin{equation}
\dot{x}=\nabla\Lambda_{t_{0}}^{t_{0}+\bar{\tau}_{0}}\left(x\right)=\partial_{x}\sigma\left(x;\delta_{0},r\right)+\mathcal{O}_{1}(\varepsilon,\delta_{0}),\label{eq:persist}
\end{equation}
in the compact neighborhood $U$ of $\mathcal{M}$, with $\mathcal{O}_{1}(\varepsilon,\delta_{0})$
denoting terms that are $\mathcal{O}(\varepsilon,\delta_{0})$ $C^{1}$-small.

Then, by Proposition \ref{prop:ridgepers0}, the dynamical system
\eqref{eq:persist} admits a ridge $\widetilde{\mathcal{M}}$ in the sense of Definition \ref{def:ridge},
which is $\mathcal{O}(\varepsilon,\delta_{0})$
$C^{1}$-close to $\mathcal{M}$, as claimed.

\bibliographystyle{plain}

\end{document}